\documentclass{article}
\usepackage{amsmath,amsthm,amssymb,graphicx}
\usepackage{color}

\newtheorem{theorem}{Theorem}[section]
\newtheorem{lemma}[theorem]{Lemma}

\newtheorem{proposition}[theorem]{Proposition}
\newtheorem{definition}[theorem]{Definition}

\newtheorem{corollary}[theorem]{Corollary}

\newtheorem*{theorem*}{Theorem}

\theoremstyle{remark}
\newtheorem{remark}[theorem]{Remark}
\newtheorem{example}[theorem]{Example}

\begin{document}

\title{The category of cellular resolutions}

\author{Laura Jakobsson\footnote{\tt laura.p.jakobsson@aalto.fi} }

\date\today

\maketitle

\begin{abstract}
Cellular resolutions are a well-studied topic on the level of single resolutions and certain specific families of cellular resolutions. One question coming out of the work on families is to understand the structure of cellular resolutions more generally. 
We give a starting point to understanding higher level structures by defining the category of cellular resolutions. In this paper we study the properties of this category. The main results are in lifting homotopy colimits from topology and Morse theory on cellular resolution being compatible with the category. 
\end{abstract}

\section{Introduction}

Cellular resolutions were first introduced by Bayer and Sturmfels in \cite{def} in order to study monomial modules. 
Earlier work of  Bayer, Peeva and Sturmfels \cite{bps} introduced the concept for simplicial cases. 
 Cellular resolutions have turned out to be a strong tool for resolving monomial modules and they are now a standard tool in combinatorial commutative algebra, and thus covered in the book by Miller and Sturmfels \cite{cca}, for example.  The definition of cellular resolutions with cell complexes brings in topology and also gives them a combinatorial nature, so we know very well how to compute them. 
A lot of the literature on cellular resolutions either cover a particular type of monomial ideals, for example Dochtermann and Mohammadi construct cellular resolutions from mapping cones in \cite{DM}, or are very involved with minimality of the resolution. It is known that every monomial module has a non-minimal cellular resolution, but in \cite{vel} Velasco showed that not all of them have a minimal cellular resolution.

Despite all the known facts about cellular resolutions, they have not been studied as a class of objects. There has been a discussion on the general structure of cellular resolutions, see for example \cite{DM} for open question on "moduli space" for a family of cellular resolutions, and even these cases are often focused on the structure of the particular family of cellular resolutions. A natural question would be to ask how do cellular resolutions behave in a more category-theoretic setting. 
This approach is also supported by the existing conversation on higher structrures on cellular resolutions, and that category theory is a fundamental tool in studying these in other fields like algebraic geometry and representation stability. 
In this paper, we define the category of cellular resolutions and study some of its properties.
This is not only interesting in its own right, but also can help us to understand cellular resolutions in general. Studying subcategories opens up a novel way to study specific types of cellular resolutions, and category theoretic constructions give us new ways to build cellular resolutions from the existing ones.

We start by generalizing the definition of cellular resolutions to cases where the cell complexes may not be connected, and then continue by defining what a map between two cellular resolutions is. For this we need the concept of compatible cellular and chain map which says that  ``they both do the same thing".
Our main result, in Definition \ref{DEF} and Theorem \ref{cat}, is the definition of the category of cellular resolutions, {\bf CellRes}, and that it does indeed form a category. 

\begin{theorem*}
{\bf CellRes} with objects being cellular resolutions and their direct sums, and morphisms being pairs $({\bf f},f)$ of compatible chain maps and cellular maps, is a well-defined category.
\end{theorem*}

In Sections 4, 5, 6 and 7 we study the common constructions in {\bf CellRes}, and note other worthwhile observations. These include mapping cones and cylinders, (co)products and (co)limits.
Throughout these sections, we see the repeating pattern of well-behaved constructions for cellular resolutions if topological and  chain complex constructions are essentially the same. Otherwise, they may not even exist in the category {\bf CellRes} in general. 

In Section 8 we turn our attention to homotopy colimits. They are a well known construction in topology, and we show that the explicit construction lifts to {\bf CellRes}. 
\begin{theorem*}
Homotopy colimits lift from topology to cellular resolutions
\end{theorem*}
In particular, this gives us a good way to construct  explicit cellular resolutions from known ones. 

In the final Section 9, we focus on discrete Morse theory on cellular resolutions. The interest in minimality has also motivated the application of discrete Morse theory to cellular resolutions in earlier work, and one example is \cite{MAG} where it is shown how to make a resolution closer to a minimal one.  We show that the algebraic Morse theory and the discrete Morse theory for cellular resolutions work well together. Our main result from this section is the following:
\begin{theorem*}
Let $\mathcal{F}$ be a cellular resolution with a cell complex $X$, and let $M$ be a Morse matching on them. 

Let ${\bf f}$ be the chain map from $\mathcal{F}$ to $\tilde{\mathcal{F}}$, and let $f$ be the cellular strong deformation retract of $X$ coming from the Morse theory.

Then the pair $({\bf f},f)$  formed of the Morse maps is a morphisms in {\bf CellRes}.
\end{theorem*}
This result shows that Morse maps are well behaved with respect to algebra and topology on cellular resolutions. Furthermore, the results on Morse theory gives a basis for simple homotopy theory for cellular resolutions. As the last result we define a simple homotopy equivalence of cellular resolutions. 

This work serves as stepping stone for further questions of categorical nature. In particular, it opens up  cellular resolutions to representation stability in the sense of Sam and Snowden \cite{SS}, and this was one of our main motivations for writing this paper. We would like to apply the representation stability results presented in the work of Sam and Snowden to {\bf CellRes}, and for this we need to have the cellular resolutions as a category. In particular, this includes the results on noetherianity and finite generation of representations,  for example Theorem 1.1.3 of \cite{SS}, applied to representations of cellular resolutions. The full details of the representation stability aspects of cellular resolutions will be made available shortly in our paper that is currently in progress. 

\subsection*{Acknowledgements}
I would like to thank Alexander Engstr\"om for many helpful discussions and guidance.

\section{Background}
In this section we review the existing tools and definitions that we will need later. 
\subsection{Category theory}
\label{catsection}
There are many good references for introductory category theory, and our main references are  \cite{Bor1}, \cite{Bor2} and \cite{ml}. 

One of the most important definitions that we need from category theory is the definition of a (locally small) category itself.
\begin{definition}
 A (locally small) \emph{category} $\mathcal{C}$ consists of a class of objects $\operatorname{obj}(\mathcal{C})$ and a set of morphisms $\mathcal{C}(a,b)$ for each pair of objects $a,b$. For any triple $a,b,c$ we have composition map of the morphisms $\mathcal{C}(b,c)\times \mathcal{C}(a,b)\rightarrow \mathcal{C}(a,c)$, with the image of the pair $(\phi,\psi)$ denoted by $\phi\circ\psi$.
The category $\mathcal{C}$ must satisfy the following two conditions:

\begin{itemize}
\item[1.] For any object $a\in obj(\mathcal{C})$ there exists an identity morphism $id_a\in \mathcal{C}(a,a)$ such that $id_a\circ\phi=\phi$ and $\psi\circ id_a=\psi$.
\item[2.] Composition of morphisms is associative, that is $(\phi\circ\psi)\circ\chi=\phi\circ(\psi\circ\chi)$ for all $\psi$ and $\phi$.
\end{itemize}
We also require that the morphism sets $\mathcal{C}(a,b)$ and $\mathcal{C}(c,d)$ are disjoint unless $a=c$ and $b=d$. 

We say that a category $\mathcal{C}$ is \emph{small} if the objects and morphisms form a set.
\end{definition}

Common examples of categories include {\bf Set}: where the objects are sets and the morphisms are set maps, {\bf Top}: objects are topological spaces and morphisms are continuous maps, and ${\bf Mod}_R$: objects are modules over the ring $R$ and morphisms are module homomorphims.

Subcategories of the category of cellular resolutions are briefly discussed in the Section \ref{subcat}, for that we have the defintion of a subcategory below.
\begin{definition}
A \emph{subcategory} $\mathcal{C}'$ of $\mathcal{C}$ is a category where $\operatorname{obj}(\mathcal{C}')\subseteq\operatorname{obj}(\mathcal{C})$ and morphisms of $\mathcal{C}$ such that the source, target, and composition are the same as in $\mathcal{C}$.
A subcategory $\mathcal{C}'$ of $\mathcal{C}$ is \emph{full} if $\mathcal{C}'(a,b)=\mathcal{C}(a,b)$ for any pair $a,b\in obj(\mathcal{C}')$.
\end{definition}

Next we define some common category theory concepts that are needed to study the basic properties of the category of cellular resolutions.
\begin{definition}
An object $a\in\mathcal{C}$ is said to be an \emph{initial object} if for all objects $b\in obj(\mathcal{C})$ there is a single morphism $a\rightarrow b$. Similarly we say that $a$ is a \emph{final object} if there is a unique morphism $b\rightarrow a$ for all $b\in obj(\mathcal{C})$. If the initial and final objects exists, they are unique up to an isomorphism.
\end{definition}

The product and coproduct constructions play a significant role for the properties that the category of cellular resolutions has and thus we state them here. For examples of product and coproduct we have them for the categories {\bf Top} and  $\mathcal{C}_{\bullet}(\bf{Mod}_S)$ in Section \ref{topcc}.
\begin{definition}
A \emph{product} of two objects $A,B$ in the category $\mathcal{C}$ is an object $A\times B$ such that there exist morphisms $f:A\times B\rightarrow A$ and $g:A\times B\rightarrow B$ and they satisfy that for any object $Z$ mapping both to $A$ and $B$ there exists a unique morphism $Z\rightarrow A\times B$ that makes the product diagram commute.
\end{definition}
\begin{definition}
A \emph{coproduct} of two objects $A,B$ in the category $\mathcal{C}$ is an object $A\sqcup B$ in $\mathcal{C}$ such that there exist morphisms $f:A\rightarrow A\sqcup B$ and $g:B\rightarrow A\sqcup B$, and they satisfy that for any object $Z$ where both $A$ and $B$ map to, there exists a unique morphism $A\sqcup B\rightarrow Z$ that makes the product diagram commute.
\end{definition}
The product and coproduct diagrams are shown in Figure \ref{prodcoprod}.
If a product or a coproduct exist, then they are unique up to unique isomorphism.
\begin{figure}
\begin{center}
\includegraphics[scale=1]{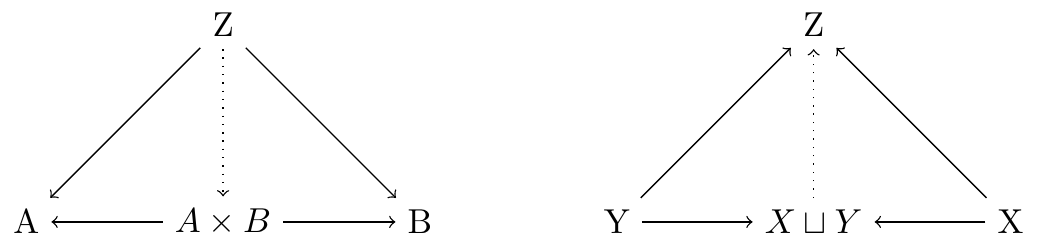}
\end{center}
\caption{The diagrams for product and coproduct in a category $\mathcal{C}$.}
\label{prodcoprod}
\end{figure}

One of the most fundamental definitions in category theory is the definition of a functor and we give this below.
\begin{definition}
A map $F$ between two categories $\mathcal{C}$ and $\mathcal{D}$ is called a \emph{(covariant) functor} and consists of a map $F:obj(\mathcal{C})\rightarrow obj(\mathcal{D})$, and for all pairs $a,b\in obj(\mathcal{C})$ there is a map $F: \mathcal{C}(a,b)\rightarrow \mathcal{C}(F(a),F(b))$.
The functor $F$ must also satisfy $F(\phi\circ\psi)=F(\phi)\circ F(\psi)$ and $F(id_a)=id_{F(a)}$.

A \emph{contravariant functor} is functor that has  a map $F: \mathcal{C}(a,b)\rightarrow \mathcal{C}(F(b),F(a))$ for all pairs $a,b\in obj(\mathcal{C})$ .
\end{definition}

\begin{definition}
 A \emph{natural transformation} $\eta$  between two functors $F,G:\mathcal{C}\rightarrow\mathcal{D}$ is a collection of maps $\{\eta_a:F(a)\rightarrow G(a)\}_{a\in obj(\mathcal{C})}$ in $\mathcal{D}$ such that the diagram
$$\begin{array}{ccc}
	F(a)&\xrightarrow{\eta_a} & G(a)\\
	
	\downarrow &&\downarrow \\

	F(b)&\xrightarrow{\eta_b}& G(b)
\end{array}$$
commutes for any morphisms $\phi:a\rightarrow b$ in $\mathcal{C}$.  The functors $F$ and $G$ are said to be isomorphic if $\eta_a$ is an isomorphism for all $a$, and $\eta$ is called a \emph{natural isomorphism}.
\end{definition}

The next few definitions cover the limits and colimits in the category setting. First we define what is a diagram in a category and then proceed to state the definitions of limit and colimit. 

\begin{definition}
A \emph{diagram} $D$ in a category $\mathcal{C}$ is a covariant functor $F:I\rightarrow\mathcal{C}$ where $I$ is a small category. $F_i$ denotes the image of $i\in obj(I)$, and for any $\phi:i\rightarrow i'$ we have  a map $F(\phi):F_i\rightarrow  F_{i'}$.
\end{definition}

\begin{definition}
A \emph{limit} of the diagram $D$ is an object $\mathrm{lim }\ M$ with maps $f_i:\mathrm{lim}\ M\rightarrow M_i$, satisfying $f_i= M(\phi)\circ f_j$ for all $\phi:i\rightarrow j$ in $\mathcal{I}$, and for any $W\in obj(\mathcal{C})$ and any family of maps $t_i:W\rightarrow M_i$ such that $t_i= M(\phi)\circ t_j$ for all $\phi:i\rightarrow j$ in $\mathcal{I}$, there exists a morphism $t: W\rightarrow\mathrm{lim}\ M$ such that $t_i=f_i\circ t$ for any object $i\in\mathcal{I}$.
\end{definition}

\begin{definition}
A \emph{colimit} of a diagram $D$ in $\mathcal{C}$ is an object $\mathrm{colim}\ M$ in $\mathcal{C}$ with a map $\iota_i: M_i\rightarrow \mathrm{colim}\ M$. The colimit must satisfy $\iota_i=\iota_j\circ M(\phi)$ for all $\phi:i\rightarrow j$ in $\mathcal{I}$, and for any $W\in obj(\mathcal{C})$ and any family of maps $t_i:M_i\rightarrow W$ such that $t_i=t_j\circ M(\phi)$ for all $\phi:i\rightarrow j$ in $\mathcal{I}$, there exists a morphism $t:\mathrm{colim}\ M\rightarrow W$ such that $t_i=t\circ \iota_i$ for any object $i\in\mathcal{I}$.
\end{definition}
If limits and colimits exist, they are unique up to isomorphism. A common examples of the two include $p$-adic numbers for the colimit and products for the limit.

The next definition is useful for the tensor product that we have for cellular resolutions, see Section \ref{tprod}.
\begin{definition}
We say that $\mathcal{C}$ is a \emph{monoidal category} if it has a bifunctor $\otimes:\mathcal{C}\times\mathcal{C}\rightarrow\mathcal{C}$, an object $e$, a natural isomorphism $\alpha:(-\otimes -)\otimes-\rightarrow -\otimes(-\otimes -)$, and natural isomorphisms $\lambda:(e\otimes-)\rightarrow-$ and $\rho:(-\otimes e)\rightarrow -$, such that they satisfy the triangle equality
$$\rho_x\otimes 1_y(x,e,y)=(1_x\otimes \lambda_y)\circ \alpha(x,e,y)$$
and the pentagon identity
$$\alpha\otimes 1\circ\alpha\circ 1\otimes\alpha(x,y,z,w)=\alpha\circ\alpha(x,y,z,w).$$
\end{definition}

The remaining definitions in this section are needed for homotopy colimits and simplicial set enrichment for later on.

\begin{definition}
Let $\mathcal{C}$ be a category. Then the \emph{opposite category} $\mathcal{C}^{op}$ is the category with the objects of $\mathcal{C}$ and morphism $b\rightarrow a$ for every $a\rightarrow b\in \mathcal{C}$.
\end{definition}

\begin{definition}
\label{enrich}
Let $\mathcal{V}$ be a monoidal category. Then a category \emph{$\mathcal{C}$ enriched with $\mathcal{V}$} is the category with objects $\textrm{obj}(\mathcal{C})$, and for every pair of objects an object $\mathcal{C}(a,b)\in\mathcal{V}$. We also have that for any triple $a,b,c\in\mathcal{C}$, we have the composition $\mathcal{C}(a,b)\otimes\mathcal{C}(b,c)\rightarrow \mathcal{C}(a,c)$. Finally, the following diagrams must commute for the given data.

$$\begin{array}{ccc}
(\mathcal{C}(a,b)\otimes\mathcal{C}(b,c))\otimes\mathcal{C}(c,d)&\longrightarrow & \mathcal{C}(a,c)\otimes \mathcal{C}(c,d)\\
\downarrow& & \\
\mathcal{C}(a,b)\otimes(\mathcal{C}(b,c)\otimes\mathcal{C}(c,d))& &\downarrow\\
\downarrow& & \\
\mathcal{C}(a,b)\otimes \mathcal{C}(b,d)&\longrightarrow & \mathcal{C}(a,d)
\end{array}$$
and 
$$\begin{array}{ccccc}
I\otimes\mathcal{C}(a,b)&\rightarrow& \mathcal{C}(a,b)&\leftarrow& \mathcal{C}(a,b)\otimes I\\
\downarrow& &\downarrow& &\downarrow\\
\mathcal{C}(a,a)\otimes \mathcal{C}(a,b)&\rightarrow & \mathcal{C}(a,b)&\leftarrow &\mathcal{C}(a,b)\otimes\mathcal{C}(b,b)
\end{array}.$$
\end{definition}

An example of the enriched category is {\bf Top} with simplicial sets (defined below). This is important example as the category of cellular resolutions inherits the structure (Proposition \ref{enri}). 
\begin{definition}
Let $\Delta$ be the category of finite sets $[n]=\{0,1,\ldots,n\}$ as the objects and order preserving functions a morphisms.
A \emph{simplicial set} is a contravariant functor $X:\Delta\rightarrow {\bf Set}$. The category of simplicial sets is denoted by {\bf sSet}.
\end{definition}

\begin{definition}
Let $\mathcal{V}$ be a closed monoidal category.
In a $\mathcal{V}$-enriched category $\mathcal{C}$, the \emph{copower} of $x\in\mathcal{C}$ by an object $v$ of $\mathcal{V}$ is an object $v\odot x\in\mathcal{C}$ with natural isomorphism $\mathcal{C}(v\odot x,y)\cong\mathcal{V}(v,\mathcal{C}(x,y))$.
\end{definition}

The nerve of the under category appears in the definition of the homotopy colimit,  and these two concepts are defined below.
\begin{definition}
Let $\mathcal{C}$ be a category and $c\in\mathcal{C}$ an object. Then the \emph{under category}, or category of objects of $\mathcal{C}$ under  $c$, $(c\downarrow\mathcal{C})$ is a category with objects $(b,f)$ where $b\in\mathcal{C}$ and $f:c\rightarrow b$, and the morphisms $(b,f)\rightarrow(b',f')$ is a map $g:b\rightarrow b'$ that makes the triangle below commute.

$$\begin{array}{ccc}
 &c& \\
 \swarrow & &\searrow\\
 b\ \ \ &\longrightarrow&\ \ \ b'
\end{array}$$
\end{definition}

\begin{definition}
Let $\mathcal{C}$ be a small category. The \emph{nerve} of $\mathcal{C}$ is the simplicial set $N\mathcal{C}$ where the $n$-simplex $\sigma$ is a diagram in $\mathcal{C}$ of the form 
$c_0\rightarrow c_1\rightarrow\ldots\rightarrow c_n$
with maps
$d_i:NC_n\rightarrow NC_{n-1}$ by composing at $i$-th object, and $s_i:NC_n\rightarrow NC_{n+1}$ by adding an identity morphisms at $i$.
\end{definition}
\subsection{Categories topological spaces and chain complexes}
\label{topcc}
\subsubsection{Topological spaces}
\begin{definition}
The \emph{category of topological spaces}, denoted by {\bf Top},  is a category that has topological spaces as the objects and for any two spaces $X,Y$ the set of morphisms ${\bf Top}(X,Y)$ consists of all continuous maps between $X$ and $Y$.
\end{definition}

The category {\bf Top} has an initial object, the empty space, as there is a continuous map from the empty space to any other topological space. 
The products in the category {\bf Top} are just the usual products of topological spaces, where the underlying space is a Cartesian product and it has the product topology.
The coproducts in {\bf Top} are disjoint unions of topological spaces.

Limits and colimits in {\bf Top} are lifted from the category of sets, that is the limit  of the diagram $D$ is the set of the set limit of the diagram with initial topology and final topology in the case of colimit.We know that all finite limits and colimits exist. 

\begin{definition}
Let $f:X\rightarrow Y$ be a continuous map. Then the \emph{mapping cone} of $f$, denoted with $C_f$, is the space $(X\times[0,1])\sqcup_f Y$ with the identification of $X\times {0}$ with a single point and $(x,1)\sim f(x)$.

The \emph{mapping cylinder} is constructed in the same way, but instead identifying $X\times{0}$ with a single point we identify it with $X$.
\end{definition}

\begin{definition}
Let $f,f':X\rightarrow Y$ be morphisms in {\bf Top}. Then $f$ is  said to be \emph{homotopic} to $f'$, denoted by $f\sim f'$, if there exists a morphism $F:X\times [0,1]\rightarrow Y$ such that $F(x,0)=f(x)$ and $F(x,1)=f'(x)$ for all $x\in X$. 
Two spaces $X$ and $Y$ are homotopic if we have morphisms $f:X\rightarrow Y$ and $g:Y\rightarrow X$ such that $f\circ g\sim\operatorname{id}$ and $g\circ f\sim\operatorname{id}$. 
\end{definition}

In {\bf Top} the colimits do not preserve homotopy, however this is a desirable property so one can define the homotopy colimit in {\bf Top}.
Homotopy colimits are defined using the category {\bf Ord} as follows, see \cite{WZZ} for more details. 
The category {\bf Ord} consists of finite sets $[n]=\{0,1,\ldots,n\}$ as the objects and non-decreasing maps, i.e. $f:[n]\rightarrow [m]$ then $f(i)\leq f(i+1)$ as the morphisms. The morphisms in {\bf Ord} are generated by two maps, namely $\delta_n^i:[n]\rightarrow[n-1]$ and $\sigma_n^i:[n]\rightarrow[n+1]$.

\begin{definition}
A \emph{simplicial space} is a contravariant functor $F$ from {\bf Ord} to {\bf Top}. 
The functors form a category of simplicial spaces with the morphisms being the natural transformations between the functors.
\end{definition}
A particular case of the simplicial space is the simple geometric realization functor $R:\bf{Ord}\rightarrow\bf{Top}$ taking the set $[n]$ to the standard $n$-dimensional simplicial complex $\Delta_n$. 

\begin{definition}
The \emph{geometric realization} of a simplicial space $F$ is the direct sum $\bigsqcup F_n\times  \Delta_n$ quotiented out by the relations $(d^i(x),p)\sim(x,R(\delta^i)(p))$ and $(s^i(x),p)\sim(x,R(\sigma^i)(p))$ where $d^i$ and $s^i$ are the images of $\delta^i$ and $\sigma^i$ under $F$.
\end{definition}
\begin{definition}
The \emph{classifying space} of a category $\mathcal{A}$ is the geometric realization of the simplicial space $F_{\mathcal{A}}$ associated to $\mathcal{A}$, which is the functor $F_{\mathcal{A}}:\bf{Ord}\rightarrow\bf{Set}$ taking the set $[n]$ to the sequence $\alpha_n\leftarrow\ldots\leftarrow\alpha_0$.
\end{definition}

For some small category $A$ and objects, let $A_{\downarrow a} $ be the category of all arrows $a\rightarrow b$ with commutative triangles as the morphisms. Let $B(A_{\downarrow a})$ be the classifying space of $A_{\downarrow}$. 
\begin{definition}
\label{tophocolim}
The \emph{homotopy colimit} of the diagram $D:A\rightarrow \bf{Top}$, denoted by $\mathrm{hocolim} D$ is the quotient of the coproduct $\sqcup_{a\in A}B(A_{\downarrow a})\times D_a$. The equivalence relation $\sim$ for the quotient is the transitive closure of $\alpha(p,x)\sim\beta(p,x)$ , where $\alpha$ and $\beta$ are the following maps 
$$\alpha:B(A_{\downarrow b})\times D_a\rightarrow B(A_{\downarrow b})\times D_b,\ \alpha(p,x)=(p,d_{f}(x)),$$
$$\beta:B(A_{\downarrow b})\times D_a\rightarrow B(A_{\downarrow a})\times D_a,\ \alpha(p,x)=(p,d_{f}(x))$$
for all morphisms $f:a\rightarrow b$.
\end{definition}

One can also approach the homotopy colimit from a more concrete view and take it as "gluing in mapping cylinders" to the diagram. 

\begin{definition}
The homotopy category of {\bf Top} is the category where the objects are same as in {\bf Top}, but the morphisms are homotopy classes of the morphisms. 
\end{definition}

\subsubsection{Chain complexes of $S$-modules}
All our rings are commutative and we reserve the notation $S$ for a polynomial ring, that is $S=R[x_1,x_2,\ldots,x_n]$, where $R$ is a commutative ring or a field. 
As with other review sections, there are many possible references and we refer the reader to \cite{Weibel} for more complete introduction to chain complexes of modules.

\begin{definition}
Let $\bf{Mod}_S$ denote the category of $S$-modules, where the objects are $S$-modules and the morphisms between a pair of modules $M$ and $N$, denoted by ${\bf{Mod}_S}(M,N)$, are the set of $S$-module homomorphisms from $M$ to $N$.
\end{definition}

\begin{definition}
 The \emph{category of chain complexes} $\mathcal{C}_{\bullet}(\bf{Mod}_S)$ is the category with the objects being chain complexes of objects of the category $\bf{Mod}_S$
$$\mathcal{C}_{\bullet}:\ \ldots\leftarrow C_0\leftarrow C_1\leftarrow\ldots\leftarrow C_n \leftarrow\ldots$$
where $C_i$ is in $\bf{Mod}_S$ and the maps $d_k:C_k\rightarrow C_{k-1}$ such that $d_id_{i+1}=0$.
The morphisms are given by chain maps between complexes. A chain map $f$ from $\mathcal{C}_{\bullet}$ to $\mathcal{D}_{\bullet}$ is a collection of maps $\{f_i:C_i\rightarrow D_i\}$ such that all the squares commute
$$\begin{array}{ccccc}
\cdots\leftarrow &C_i& \leftarrow& C_{i+1}&\leftarrow\cdots\\
 &\downarrow & & \downarrow&\\
 \cdots\leftarrow &D_i& \leftarrow& D_{i+1}&\leftarrow\cdots\\
\end{array}.$$
\end{definition}
\begin{remark}
We have stated the definition for the category of chain complexes of $S$-modules, however chain complexes can be defined for any category.
\end{remark}

In the category of chain complexes of $S$-modules the product is given by 
 the direct sum of two complexes, so the direct sum product of $C$ and $D$ is $C\oplus D$ with $(C\oplus D)_k=C_k\oplus D_k$ in the finite case. In the case of finite coproducts it is also the direct sum.
Limits and colimits can be computed degree wise in the category of chain complexes, and the category is also additive degree wise. From the degree wise property we have an explicit description of the limit and colimit in the category given by the following definition:

\begin{definition}
Let $f,g:\mathcal{C}_{\bullet}\rightarrow\mathcal{D}_{\bullet}$ be two chain maps. A \emph{homotopy} between $f$ and $g$ is a collection of maps $h_i:C_i\rightarrow D_{i+1}$ such that
$$f_i-g_i=d_{i+1}\circ h_i+h_{i-1}\circ d_i.$$
If a collection of the maps $h_i$ exists, then we write $f\sim g$.
Two complexes $\mathcal{C}_{\bullet}$ and $\mathcal{D}_{\bullet}$ are said to be \emph{homotopy equivalent}, denoted by $\mathcal{C}_{\bullet}\simeq\mathcal{D}_{\bullet}$, if there are chain maps $f:\mathcal{C}_{\bullet}\rightarrow\mathcal{D}_{\bullet}$ and $g:\mathcal{D}_{\bullet}\rightarrow\mathcal{C}_{\bullet}$ such that $f\circ g\sim\operatorname{id}$ and $g\circ f\sim\operatorname{id}$.
\end{definition}

 \begin{definition}
 Let $C$ and $D$ be two chain complexes, then the \emph{tensor product} $C\otimes D$ is given by
$$(C\otimes D)_k=\bigoplus_{i+j=k}C_i\otimes D_j$$
with differential  
$$d_k(x\otimes y)=d_i^C(x)\otimes y+(-1)^i x\otimes d_j^D(y) \textrm{ where }x\in C_i, y\in D_j.$$
\end{definition}

\begin{definition}
Let $\psi:G\rightarrow F$ be a map of chain complexes. Then the \emph{mapping cone} of $\psi$, $C(\psi)$, is the chain complex 
$$C(\psi)_i=F_i\oplus G_{i-1}$$
with differential map
$$d_{i}(f,g)=(-\psi(g)+d(f),-d(g)).$$
\end{definition}

As with topological spaces, we also have a mapping cylinder of chain complexes.
\begin{definition}
Let $\psi:G\rightarrow F$ be a map of chain complexes. Then the \emph{mapping cylinder} of $\psi$, $C(\psi)$ is the chain complex 
$$C(\psi)_i=F_i\oplus G_i\oplus G_{i-1}$$
with differential map
$$d_{i}(f,g,g')=(-\psi(g)+d(f),d(g)+id(g'),-d(g'))$$
\end{definition}

\subsection{Simplicial and CW-complexes}
Simplicial and CW-complexes are covered by many standard topology books, for example \cite{Schubert}.

\begin{definition}
An \emph{abstract simplicial complex} $\Delta$ is a set of vertices $V=\{1,\ldots,n\}$ with collection of subsets $A$ of $V$ such that if $A\subseteq V$ and $B\subseteq A$ then $B\in V$. The subsets are called simplices and we have that $\operatorname{dim}A=|A|-1$. The dimension of the simplicial complex $\Delta$ is the maximum dimension of its simplices. A face of $A$ in $\Delta$ is a nonempty subset $B\subseteq A$.
\end{definition}

\begin{definition}
A \emph{d-cell} is a topological space that is homeomorphic to the closed unit ball $B^d$ in $d$-dimensional Euclidean space. Let $\sigma$ be a d-cell, then $\sigma'$ denotes the subset corresponding to the $d-1$ sphere in $B^d$. By a \emph{cell} we mean a topological space that is a $d$-cell for some $d$.
\end{definition}

Let $X$ be a topological space and $\sigma$ a $d$-cell, and let
$f:\sigma'\rightarrow X$
be a continuous map. Then one can attach $\sigma$ to $X$ by taking the disjoint union $X\cup_f\sigma$, where $\sigma$ is quotiented by the relation identifying $x\in\sigma'$ with $f(x)$. The map $f$ is called attaching map in this case.

\begin{definition}
Any topological space $X$ is a {\it finite CW-complex} if it has a finite sequence $\emptyset\subset X_0\subset X_1\subset\ldots\subset X_n=X$ such that each $X_i$ is a result of attaching a cell to $X_{i-1}$. Note that this requires $X_0$ to be a 0-cell.

The sequence is called the {\it CW-decomposition} of $X$.
\end{definition}

A $d$-simplex with a geometric realization is a $d$-cell, hence we get that simplicial complexes are also CW-complexes.
The CW-complexes form a subcategory of {\bf Top}, and inherit the basic constructions defined for {\bf Top}. 

\begin{definition}
Let $X$ and $Y$ be CW-complexes. Then the \emph{join} of $X$ and $Y$ is the complex w get by connecting every vertex of $X$ to all vertices of $Y$ with an edge, and filling in the higher degrees accordingly.
\end{definition}

\begin{definition}
Let $X$ and $Y$ be CW-complexes. A continuous map $f:X\rightarrow Y$ is \emph{cellular} if $f(X_n)\subseteq Y_n$.
\end{definition}

Next we state a well-known theorem, that is found in many books. See \cite{Schubert} for a proof.
\begin{theorem}[Cellular approximation theorem]
Any map between CW-pairs is homotopic to a cellular map.
\end{theorem}

\begin{definition}
We say that the CW-complex is regular if each of the $X_i$, for all $i$, is homeomorphic to a ball. 
\end{definition}
Regular complexes have geometric properties that are beneficial and in particular the properties of 2 and 3 from the Proposition \ref{idk} are needed for well-behaving cellular resolutions. We state these below as a proposition.
\begin{proposition}[\cite{cf}, Chapter 2]
\label{idk} Let $X$ be a regular CW-complex and $\sigma_n$ an n-cell of $X$, then
\begin{enumerate}
\item If $m<n$ and $\sigma_m$ and $\sigma_n$ are cells such that their intersection is non-empty, then we have that $\sigma_m\subset\sigma_n$.
\item  For $n\geq0$, $\overline{\sigma_n}$ and $\partial \sigma_n$ are subcomplexes, and furthermore $\partial\sigma_n$ is the union of closures of (n-1)-cells.
\item If $\sigma_n$ and $\sigma_{n+2}$ are cells such that $\sigma_n$ is a face of $\sigma_{n+2}$, then there are exactly two (n+1)-cells between them.
\end{enumerate}
\end{proposition}

\begin{definition}
Let $X$ be a regular CW-complex. Then
$X$ comes equipped with an orientation of the faces, and a function $\operatorname{sign}(F',F)$ on pairs of faces $F,F'$. The functions take values in $\{0,1,-1\}$, with $\operatorname{sign}(F',F)$ nonzero if and only if $F'$ is a facet of $F$, and $\operatorname{sign}(F',F)=1$ if the orientation of $F'$ induces the orientation for $F$.
\end{definition}
 The $\operatorname{sign}(F',F)$ can also be thought of as giving the sign of $F'$ in the boundary map of $F$.  
\begin{proposition}[\cite{sth}, Lemma 7.1]
The sign- function given above exists for regular CW-complexes and satisfies the described properties.
\end{proposition}

For polyhedral cell complexes $X$ one can associate a chain complex to it with the differential maps given by
$\partial(F)=\sum_{G\subset F}\textrm{sign}(G,F)G$. 
\begin{definition}
 A \emph{reduced chain complex} $\tilde{C}(X;k)$ for $X$ is a chain complex, where the ith vector space in the chain complex is $k^{F_i}$ with basis consisting of vectors labelled by the $i$ dimensional faces of $X$. 
\end{definition} 
\begin{remark}
 Different orientations for the cell complex give isomorphic chain complexes, and so the orientation can be chosen freely. 
\end{remark}

\subsection{Discrete and algebraic Morse theory}
We focus our attention on discrete and algebraic Morse theory due to the nature of the objects we study.

\subsubsection{Discrete Morse theory}

The main reference used for this discrete Morse theory section is \cite{forman}.

\begin{definition}
Let $X$ be a cell complex. 
A \emph{face poset} diagram $P_X$ for $X$ is a directed graph with vertices corresponding to $n$-cells of the cell complex. We have an edge from $\beta$ to $\alpha$ if and only if $\alpha$ is a codimension 1 face of $\beta$.
\end{definition}

\begin{definition}
A \emph{matching} on a graph is a set of pairwise non-adjacent edges.
Let $X$ be a cell complex with  face poset $P_X$. Then a \emph{Morse matching} on $P_X$ is a matching $M$ such that $P_X$ has no directed cycles when the edges in $M$ are reversed.

A vertex is \emph{critical} if it is not in the Morse matching.
\end{definition}

Now we can state the main theorem of Morse theory. We have chosen to use the form  given in $\cite{E10}$ since it will be convenient in the later sections.
\begin{theorem}[Main theorem of discrete Morse theory,\cite{E10}]
If $X$ is a regular CW-complex with a Morse matching (giving at least one critical vertex), then there exists a CW complex $Y$ that is homotopy equivalent to $X$, and the number of d-dimensional cells of $Y$ equals the number of d-dimensional critical cells of $X$ for every d.
\end{theorem}

A Morse matching with a single edge is an elementary collapse in discrete Morse theory. This can be explicitly on the CW-complex by the following definition, see \cite{cohen} for more details.
\begin{definition}
Let $X$ be a finite CW-complex and let $Y$ be a subcomplex of $X$.
Then there is an \emph{elementary collapse} of $X$ to $Y$, $X\searrow^e Y$ if and only if
$X=Y\cup e^{n-1}\cup e^n$ where $e^{n-1}$ and $e^n$ are not in $Y$, and there exists a ball pair $(Q^n,Q^{n-1})$ and a map $\varphi: Q^n\rightarrow X$ such that
\begin{itemize}
\item $\varphi$ is a characteristic map for $e^n$,
\item $\varphi|Q^{n-1}$ is a characteristic map for $e^{n-1}$, and
\item $\varphi(P^{n-1})\subset Y$ where $p^{n-1}=C1(\partial Q^n-Q^{n-1})$.
\end{itemize}
\end{definition}

\subsubsection{Algebraic Morse theory}
\label{algmorse}

The algebraic analogue of discrete Morse theory was developed by Sk\"oldberg \cite{ES05} and J\"ollenbeck and Welker \cite{jol} independently. It allows us to apply Morse theory techniques to chain complexes.
For a more complete and  detailed overview of algebraic Morse theory, the reader may look up the original works of Sk\"oldberg and J\"ollenbeck. The notation used in this section follows that of \cite{ES05}.

Let $K$ be a based chain complex of $S$-modules
$$0\longleftarrow K_0\xleftarrow{d} K_1\xleftarrow{d} K_2\longleftarrow\cdots$$
with $K_{i}=\bigoplus_{j}K_{i,j}$, where $K_{i,j}$ is an $S$-module and $d$ is the differential in the chain complex.

\begin{definition}
The \emph{directed graph associated to $K$}, denoted by $\Gamma_K$, is defined as follows. The vertices of the graph are given by the summands in each homological degree and the directed edges go down in the degrees. We have an edge from $K_{i,j}$ to $K_{i-1,j'}$ if $d(K_{i,j})\cap K_{i-1,j'}$ is not empty. 
We denote by $d_{j,k}$ the component of the differential corresponding to an edge from $K_{i,k}$ to $K_{i-1,j}$.
\end{definition}

\begin{remark}
The graph depends on the decomposition chosen for the $K_i$ in the chain complex. 
\end{remark}

\begin{definition}
A \emph{Morse matching} on the graph $\Gamma_K$ is a partial matching $M$ on $\Gamma_K$, satisfying that there are no directed cycles in the graph $\Gamma_K^M$, which is $\Gamma_K$ with the edges from $M$ reversed, and that the maps in $K$ corresponding to the edges in $M$ are isomorphisms. 
\end{definition}

\begin{proposition}[\cite{ES05}, Chapter 2]
\label{algmorsemap}
From the Morse matching $M$ we can form a graded map $\varphi: K\rightarrow K$. If $j$ is minimal with respect to the partial order $\prec$ and  $x\in K_{i,j}$, the map is given by
$$\varphi(x)=\left\lbrace\begin{array}{cl}
d^{-1}_{j,k}(x),& \exists\textrm{ an edge from }K_{i,k}\textrm{ to }K_{i-1,j}\textrm{ for some }k\in M\\
0,&\textrm{otherwise}
\end{array}\right.$$
If $j$ is not minimal then $\varphi$ is given by
$$\varphi(x)=\left\lbrace\begin{array}{cl}
d^{-1}_{j,k}(x)-\sum\varphi d_{m,k}d^{-1}_{j,k}(x)& \exists\ \textrm{an edge from }K_{i,k}\textrm{ to }K_{i-1,j}\textrm{for }k \in M\\
0&\textrm{otherwise}
\end{array}\right.$$
where the sum is over all edges from $K_{i,k}$ to $K_{i-1,m}$
The map $\varphi$ is a splitting homotopy as ti satisfies $\varphi^2=0$ and $\varphi\circ d\circ\varphi=\varphi$.
\end{proposition}

Let $\pi$ be the chain map given by $\pi=id-(d\circ\varphi+\varphi\circ d)$. We have that $\pi(v)=0$ if $v$ is a vertex incident to an edge in the partial matching $M$.

\begin{theorem}[\cite{ES05}, Theorem 1]
Let $M$ be a Morse matching on the complex $K$. Then the complexes $K$ and $\pi(K)$ are homotopy equivalent. Furthermore we have for each $n$ an isomorphism of modules $\pi (K_n)\cong \bigoplus_{\alpha\in M_n^0}K_{\alpha}$.
\end{theorem}

\begin{remark}
Instead of $\pi(K)$, we can look at the chain complex $\overline{K}$ given by $$\overline{K}_i=\bigoplus_{K_{ij}\textrm{ is unmatched in }M}K_{i,j}.$$ 
Let $\rho$ be the projection from $K=\bigoplus_iK_i$ to $\overline{K}$.
The differential $\overline{d}$ can be defined as $\overline{d}=\rho(d-d\varphi d)$. The complex $\overline{K}$ is then also homotopy equivalent to~$K$.
\end{remark}

\section{The category of CellRes}
We want to define the category of cellular resolutions.
In Section \ref{cellres} we give the definition of a cellular resolution. Then we  define morphism, in detail, and finally in Section \ref{def} we define the category of cellular resolutions.

\subsection{Cellular resolutions}
\label{cellres}

Most of the material in this section can be found in Miller and Sturmfels \cite{cca}. 
In \cite{cca} cellular resolutions are defined over a connected regular CW-complex. However, we see no reason to restrict ourselves to this case, rather the contrary, we want the non-connected cell complexes as well. This difference does not show up in the definition, so it is the same as found in \cite{cca}. 

\begin{definition}
 A labeled cell complex $X$ is a regular CW-complex with monomial labels on the faces. The  vertices of $X$ have labels $\bf{x}^{\bf{a}_1},\bf{x}^{\bf{a}_2},\ldots,\bf{x}^{\bf{a}_r}$ where $\bf{a}_1,\bf{a}_2,\ldots,\bf{a}_r\in\mathbb{N}^n$. The faces $F$ of $X$ have the least common multiple of the monomial labels of the vertices it contains, $x^{\bf{a}_F}=\mathrm{lcm}\{x^{\bf{a}_v}:v\in F\}$. The label on the empty face is 1, i.e. $\bf{x}^{\bf{0}}$.
\end{definition}
\begin{definition} 
  The \emph{degree} of a face $F$ is the exponent vector $\bf{a}_F$ of the monomial label.
\end{definition}

Recall that for a non-labeled cell complex we can construct the reduced chain complex of vector spaces. In the case of a labeled cell complex, we also have the algebraic data of the monomial labels, which we would like to see included in the data of the chain complex. Thus we can define the
following complex of free $\mathbb{N}^n$ graded $S$-modules, called the cellular free complex, and denoted by $\mathcal{F}_X$. 

\begin{definition}
Let $S(-\bf{a}_F)$ be the free $S$-module with a generator $F$ in degree $\bf{a}_F$. Then the \emph{cellular complex} $\mathcal{F}_X$ is given by $(\mathcal{F}_X)_i=\bigoplus_{\substack{F\in X\\ \mathrm{dim}F=i-1}}S(-\bf{a}_F)$ with a differential  $$\partial(F)=\sum_{G\subset F}\textrm{sign}(G,F)x^{\bf{a}_F-\bf{a}_G}G.$$  
\end{definition}
The differential is a homogeneous map, so it preserves the degree.

\begin{remark}
Often one only considers the coarser $\mathbb{N}$-grading for the chain complex, as in many examples no significant data is lost by this. For simplicity, we will omit grading in several examples, as often is the case.
\end{remark}

\begin{proposition}[\cite{cca}, Definition 4.3]
The differentials in the  cellular complex can also be described by monomial matrices, with the columns and rows having the corresponding faces as labels and the scalar entries coming from the usual differential for reduced chain complex. The free $S$-modules of $\mathcal{F}_X$ are then the ones represented by the matrices. 
\end{proposition}

The above $\mathcal{F}_X$ certainly defines a chain complex but for a reolution we require the chain complexes to have homology only at degree 0. With this in mind one has the following standard definition for cellular resolutions.
\begin{definition}
We call the chain complex $\mathcal{F}_X$ a \emph{cellular resolution} if it is acyclic, that is, $\mathcal{F}_X$ has non-zero homology only at degree 0.
\end{definition}
\begin{remark}
The subscript $X$ in $\mathcal{F}_X$ emphasizes from which cell complex the resolution comes from, and the subscript can be omitted at times.
Sometimes the cellular resolutions is thought of as the pair $(X,\mathcal{F})$, where $X$ is the labeled cell complex and $\mathcal{F}$ is the cellular resolution. 
\end{remark}

If the CW-complex $X$ supporting a cellular resolution is connected, then we would only get one ideal from the labels, see for example \cite{cca} and \cite{def}. 
 The main difference between the connected and unconnected case is in the part of the module resolved; we have the following multiple component version of a Proposition 4.5 from \cite{cca} that is the same as theirs in the case of $X$ being connected.
Firstly, we need the definition for sub-complexes bound by labels.
 
\begin{definition}
For ${\bf a,b}\in\mathbb{N}^n$ we have that $\bf{a}\leq\bf{b}$ if $\bf{b}-\bf{a}\in\mathbb{N}^n$. Let $X_{\preceq \bf{b}}$ be the sub-complex of $X$ given by all the faces with labels $\preceq \bf{b}$ coordinate wise. Then let $X_{\prec\bf{b}}$ be the sub-complex  with all the faces having labels $\prec\bf{b}$.
\end{definition}
\begin{proposition}
\label{resmodule}
The cellular free complex $\mathcal{F}_X$ supported on $X=\sqcup X_i$ is a cellular resolution if and only if $X_{\preceq \bf{b}}$ is acyclic over $k$ for all $\bf{b}\in\mathbb{N}^n$.
When $\mathcal{F}_X$ is a cellular resolution then it is a resolution of $$S/I_1\oplus\ldots\oplus S/I_n$$ where $I_i$ is the ideal generated by the monomial labels on the vertices of $X_i$. 
\end{proposition}
\begin{proof}
The proof follows the proof of Proposition 4.5 in \cite{cca}, and restricting to a single component recovers it.
For the cell complex $X$ consisting of disjoint cell complexes $X_i$ for $i=1,\ldots,n$ with ideals $I_i$ generated by the labels of the component $X_i$, then the cellular resolution of $X$, if it exists, is a resolution of $S/I_1\oplus\ldots\oplus S/I_n$ and also satisfies the $X_{\preceq \bf{b}}$ condition from the above theorem.
 If each of the components of $X$ satisfies the condition that $(X_i)_{\preceq \bf{b}}$, then the whole $X=\sqcup X_i$ will have that $X_{\preceq \bf{b}}$ is acyclic for all $\bf{b}\in\mathbb{N}^n$ as a direct sum of chain complexes preserves acyclicity. The image of the last map in the resolution is $I_1\oplus\ldots\oplus I_n\subset S^n$, and has $S/I_1\oplus\ldots\oplus S/I_n$ as the cokernel.\end{proof}
\begin{remark}
Not every cell complex has a labeling that gives a cellular resolution, for example a triangle consisting of only edges and no interior does not have a labeling that would give a cellular resolution.
\end{remark}

\begin{example}
A common cellular resolution is the Taylor resolution. It is defined for any finitely generated monomial ideal $I\subset S$. Suppose that $I$ has $r$ generators. The Taylor  resolution is supported on $(r-1)$-dimensional simplex where each of the the vertices is labeled with one generator. When the ideal is given by the variables of $S$, the Taylor resolution is also a Koszul complex.
\end{example}
\subsection{Morphisms: compatible maps}
There have been few occasions where maps between cellular resolutions appear in the literature. 
In \cite{DM} the construction through the mapping cone of a cellular resolution gives a map that is a lift of the multiplication by one of the generators in the ideal. 
Another map  is the Morse map, that we get from discrete Morse theory (see Section \ref{morsesec} for more details).

For the morphisms we want maps that respect both the algebraic and topological structure of the cellular resolutions. This motivates the following definitions, and Example \ref{ex1} shows why one does not choose to take the standard chain maps between cellular resolutions.

\begin{definition}
Let $g:X\rightarrow Y$ be a cellular map between two labeled cell complexes $X$ and $Y$ with label ideals $I$ and $J$, respectively. 
The set map $\varphi_g: I\rightarrow J$ is the map defined by the action of $g$, i.e. the label $m_x\in I$ maps to $m_y\in J$ if and only if the face $x$ labeled by $m_x$ maps to the faces $y_1,\ldots, y_r$ labeled by $m_{y_1},\ldots,m_{y_r}$ with $m_y=\mathrm{lcm}(m_{y_1},\ldots,m_{y_r})$ under $g$, and $m_x\in I$ maps to $0$ if and only if the face labeled by $m_x$ is not mapped to anything in $Y$. 
\end{definition}

\begin{definition}
We say that the cellular map $g:X\rightarrow Y$ is compatible with a chain map ${\bf f}:\mathcal{F}_X\rightarrow \mathcal{F}_Y$ if they satisfy the following. The equality $f_0(x)=\varphi_g(x)$ holds for all $x\in I$, and $f_i$ maps the generator $e_x$, associated to face $x\in X$, in $\mathcal{F}_{X,i}$ to some linear combination of the generators $e_{y_i}$, $i\in\{1,2,\ldots,r\}$, associated to $y_i\in Y$ with the coefficients in $S$ if and only if $g$ maps $x$ to union of $y_1,y_2,\ldots, y_r$. 
\end{definition}

\begin{remark}
Given a cellular map $g$, we can identify a single chain map {\bf f} that is compatible with it. On the other hand, a chain map {\bf f} may be compatible with multiple cellular maps.

Let ${\bf f}$ be a chain map between two cellular resolutions $\mathcal{F}_X$ and $\mathcal{G}_Y$ with label ideals $I$ and $J$, respectively.
For simplicity let us assume that both $\mathcal{F}$ and $\mathcal{G}$ are from connected CW-complexes, so the resolutions start with $S$.  The generators of $\mathcal{F}_n$ correspond to the $(n+1)$-dimensional faces of $X$. Taking the differential $d_{n+1}$ corresponds to taking the boundary of the face corresponding to some generator of $\mathcal{F}_n$.

The compatibility conditions on ${\bf f}$ imply that $f_0$ takes the generators of $I$ to the generators of $J$. This gives some information on the cellular map $g$, explicitly how it maps the vertices from one complex to another. Furthermore,
the maps $f_n$ can be thought of as corresponding to a description of which dimension $n$ face maps to where. So now we would have information on $g$ as to which face in $X$ maps to which face in $G$.
The above conditions do not define a unique map on the topological space, but rather a family of homotopic continuous maps, as how the faces map to each other is not relevant for the algebraic side.
\end{remark}

Now we can define a map between two cellular resolutions.  
\begin{definition}
Let $\mathcal{F}_{X}$ and $\mathcal{F}_Y$ be two cellular resolutions coming from labeled CW-complexes $X$ and $Y$. 
A \emph{cellular resolution map} between the two cellular resolutions is a pair of maps $({\bf f},f)$ where ${\bf f}: \mathcal{F}_X\rightarrow \mathcal{F}_Y$ is a chain map and $f: X\rightarrow Y$ is a cellular map, such that the two are compatible.  
\end{definition}

\subsubsection{Examples}
\begin{example}
\label{ex1}
Let $S=k[x,y]$, where $k$ is a field, and let $\mathcal{F}$ be the Koszul complex of $(x,y)$ and $\mathcal{G}$ be the cellular resolution supported on the same cell complex as $\mathcal{F}$ but with labels $xy$ and $xy$.
Let us consider the possible maps from $\mathcal{F}$ to $\mathcal{G}$.
If we want the maps to respect the cellular structure, that is, sending vertices to vertices, we have four possible maps. These are illustrated in the Figure \ref{ex1maps}. 

\begin{figure}
\begin{center}
\includegraphics[scale=1]{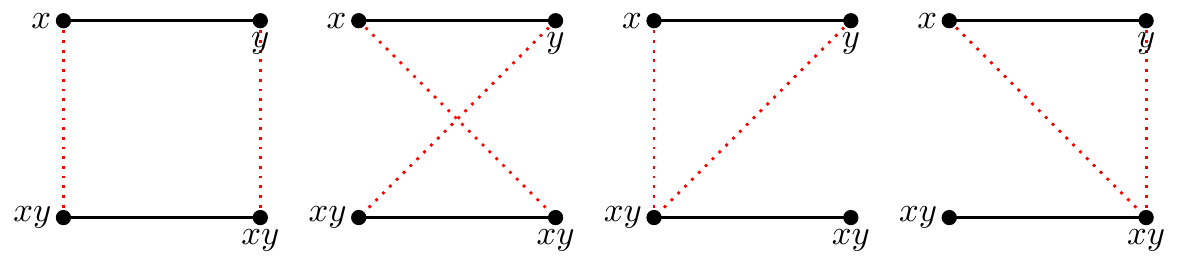}
\end{center}
\caption{Possible maps between the two cell complexes of Example \ref{ex1}.}
\label{ex1maps}
\end{figure}
On the level of resolutions the map is a chain map ${\bf f}$ where every square in the diagram of Figure \ref{ex1res} commutes.
\begin{figure}
\begin{center}
\includegraphics[scale=1]{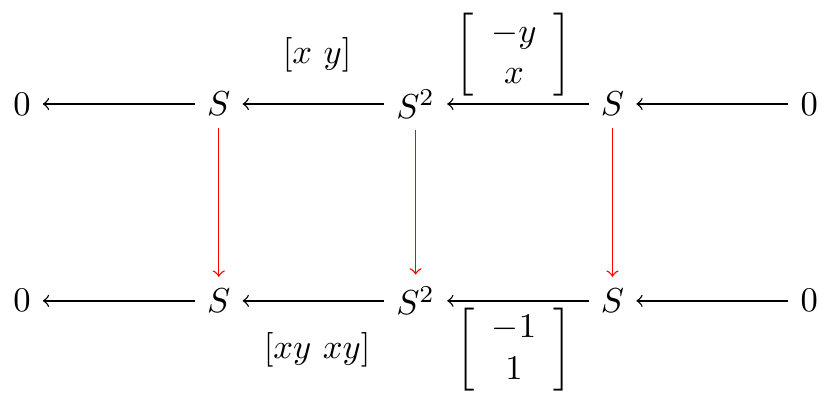}
\end{center}
\caption{The cellular resolutions of Example \ref{ex1}.}
\label{ex1res}
\end{figure}
The map $f_0$ has to match up with the mapping of the vertices, so in the first case we have that it maps $x$ to $xy$ and $y$ to the other $xy$. One can then check that the map making this commute for $f_1$ has to be one of the four maps $\left[\begin{array}{cc}
1&0\\
0&1
\end{array}\right]$,$\left[\begin{array}{cc}
0&1\\
1&0
\end{array}\right]$,$\left[\begin{array}{cc}
0&0\\
1&1
\end{array}\right]$, or $\left[\begin{array}{cc}
1&1\\
0&0
\end{array}\right]$. The first matrix maps the generators in the way we want for the case studied, and the other matrices correspond to the other three cases. It can be checked easily that for any of the four maps there is no $f_2$ that would make the second square commute, this can be done by computing the image of $\left[\begin{array}{c}
	-y\\
	x
\end{array}\right]$ composed with one of the maps and noticing that it can never be inside the image of $\left[\begin{array}{c}
	-1\\
	1
\end{array}\right]$. This implies that even if the map makes sense topologically on the level of cellular resolutions it does not work. 

However, just for the chain complexes one can find maps that behave well algebraically, for example le $f_1$ be given by $\left[\begin{array}{cc}
x&0\\
0&y
\end{array}\right]$.Then to make the squares commute one can check that $f_0$ and $f_2$ have to be multiplication by $xy$. 
Similarly we can choose $f_1$ to be given by $\left[\begin{array}{cc}
0&y\\
x&0
\end{array}\right]$ which would still have the other $f_i$ stay the same.
One may also try mapping things to a single generator, so now the map $f_1$ is $\left[\begin{array}{cc}
x&y\\
0&0
\end{array}\right]$( or  $\left[\begin{array}{cc}
0&0\\
x&y
\end{array}\right]$if one considers the last case). Again the other maps can be found by checking what maps make the squares commute, as we have that $f_1\circ d_2=0$ we get that $f_2=0$, and for $f_0$ we have that is must be the multiplication by $xy$.

None of the above maps preserve the degree of the elements between the resolutions. Trying to construct such map one would run into problems with $f_1$, as it is a map $S^2(-1)\rightarrow S^2(-2)$, so the constants in S(-1) have degree $1$, but in $S(-2)$ the only object of degree 1 is 0.  A condition that is reasonable to require from the maps is that the change in the degree is constant when the map is not a zero map, which follows from the commutativity of the squares.
\end{example}

\begin{example}

Let $S=k[x,y,z]$. Let $\mathcal{F}$ be the minimal resolution of the maximal ideal $I=(x,y,z)$ and let $\mathcal{G}$ be the minimal resolution of $I^2$.
We want to consider the map that as a cellular map sends $\mathcal{F}$ to the top rectangle of $\mathcal{G}$, as shown in Figure \ref{ex2maps}.
\label{ex2}
\begin{figure}

\begin{center}
\includegraphics[scale=1]{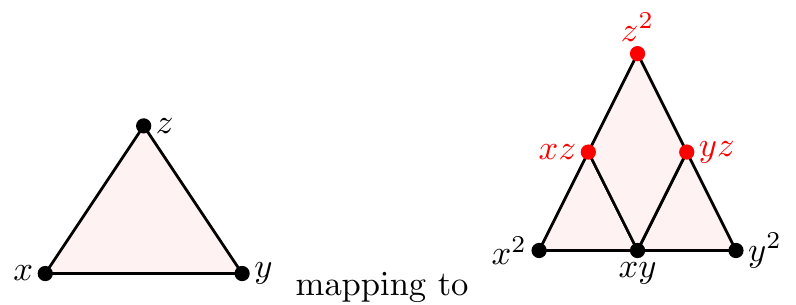}
\end{center}
\caption{Cellular map of Example \ref{ex2}.}
\label{ex2maps}
\end{figure}
The cellular resolutions of $\mathcal{F}$ and $\mathcal{G}$ are displayed in Figure \ref{ex2res}.
\begin{figure}
\includegraphics[scale=1]{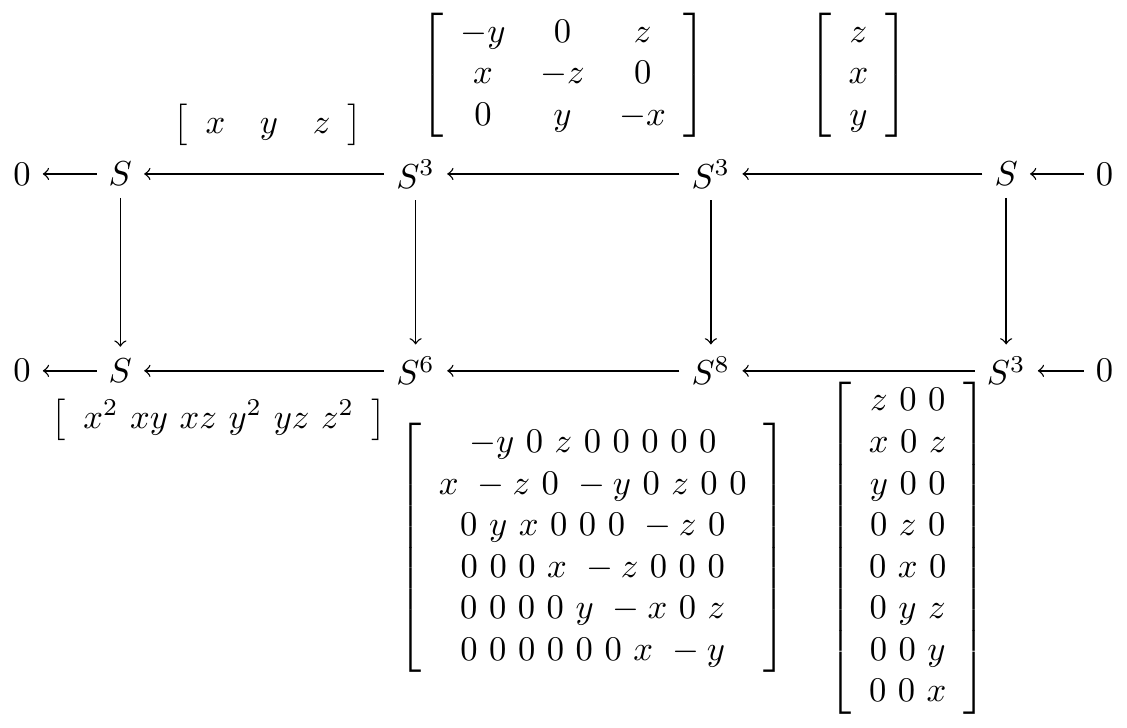}
\caption{Cellular resolutions of Example \ref{ex2}.}
\label{ex2res}
\end{figure}
On the level of the labels the map is a multiplication by $z$, so we know that the map $f_0$ is multiplication by $z$. 
Then one can choose $f_1$ such that it makes the first square commute. With a little computation one gets the matrix 
$$\left[\begin{array}{ccc}
0&0&0\\
0&0&0\\
1&0&0\\
0&0&0\\
0&1&0\\
0&0&1
\end{array}\right].$$
Clearly, as it consists of entries that are 1, and there is only one entry in each row the map $f_1$ sends the generators of $S^3$ in the first resolution to (some of) the generators of $S^6$ in the second resolution. Because of the ordering of the vertices this map corresponds to sending the vertices of the triangle to the red vertices. 
Constructing $f_2$ such that the second triangle commutes, and then $f_3$, we get the maps
$$f_2=\left[\begin{array}{ccc}
0&0&0\\
-1&0&0\\
0&0&0\\
0&0&0\\
0&0&0\\
-1&0&0\\
0&0&-1\\
0&-1&0
\end{array}\right],\ \mathrm{and} \ f_3=\left[\begin{array}{c}
0\\ 0\\ -1
\end{array}\right].$$
 These correspond to the map between cell complexes taking the edges and centre of the triangle to the edges and centre of the rectangle in the other complex. One of the edges gets subdivided and this is represented by having two entries in one column in the resolution map. 
 \end{example}
 
 \begin{example}
 \label{notmorph}
 
 Let $\mathcal{F}$ and $\mathcal{G}$ be the same cellular resolutions as in the previous example. Consider the cellular map taking $x\mapsto x^2$, $y \mapsto y^2$, and $z\mapsto z^2$.
 
 \begin{figure}
\begin{center}
\includegraphics[scale=1]{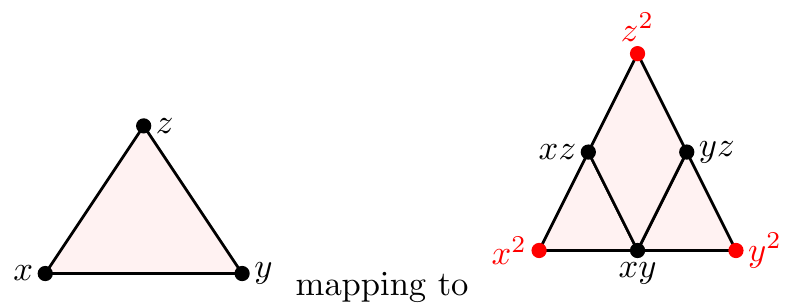}
\end{center}
\caption{Cellular map of the complexes in Example \ref{notmorph}.}
 \label{notmorphmaps}
\end{figure}

Figure \ref{notmorphmaps} shows this cellular map. If we consider the associated chain map to this, by the compatibility definition we notice that it
 does not work on the algebraic side. There is no map $f_2$ that would make the squares commute. Hence we do not have a cellular map.
\end{example}

\begin{example}
\label{projex}

Let $\mathcal{F}$ be the cellular resolution consisting of the direct sum of the two resolutions in the Example \ref{notmorph}, i.e. the cell complex is the disjoint union and the resolution is 
$$0\xleftarrow{} S\oplus S\xleftarrow{} S^3\oplus S^6\xleftarrow{} S^3\oplus S^8\xleftarrow{} S\oplus S^3\leftarrow 0.$$
Now consider the projection down to 
the Koszul complex of $I=(x,y,z)$.
We have standard projections both for the cell complex and the chain complex. More over the two are compatible as the label map generated by the topological projection is the same as the first component in the chain complex projection. 
Thus, the projection in this case is a cellular resolution morphism.
\end{example}

\subsection{The definition of  CellRes}
\label{def}

Now that we have defined the cellular resolutions and maps between them, we are ready to define the category {\bf CellRes}.

\begin{definition}
\label{DEF}
We define {\bf CellRes} to be the following data:
\begin{itemize}
\item A class of objects consisting of cellular resolutions, coming from any regular CW-complex.
\item A set of morphisms for any pair of objects $\mathcal{F}_X$ and $\mathcal{G}_Y$ with individual maps given by the compatible pairs $({\bf f},f)$.
\end{itemize}

\end{definition}

\begin{proposition}
\label{cat}
The defined data of {\bf CellRes} is a category.
\end{proposition}
\begin{proof}
With this definition of morphisms, for every pair of cellular resolutions $\mathcal{F}$ and $\mathcal{G}$ there is a set (possibly empty) of maps. There also exists an identity morphism $id_{\mathcal{F}}=({\bf id}_{\mathcal{F}},id_X)$ for every cellular resolution $\mathcal{F}$, where ${\bf id}_{\mathcal{F}}$ is the identity map of chain complexes on the resolution $\mathcal{F}$ and $id_X$ is the identity map on the cell complex $\mathcal{F}$ is supported on. We also have that the composition of the morphisms is associative in each component as both composition of chain maps and composition of cellular maps are. 
\end{proof}
\begin{remark}
The category {\bf CellRes} depends on the base ring where the labels are defined, and each ring gives a separate category. In the case the underlying ring $S$ is of importance we denote the category by $\bf{CellRes}_S$. 

For two polynomial rings $R$ and $S$ with a ring homomorphism we get a functor $\varphi: {\bf CellRes}_R\rightarrow{\bf CellRes}_S$ given by mapping the resolution $F$ in ${\bf CellRes}_R$ to the resolution $G$ in ${\bf CellRes}_S$ by the induced map on the free modules. The morphisms only change the chain map according to the induced module map and the cell map stays the same as the topological structure does not change.
\end{remark}

\subsubsection{Subcategories of CellRes} 
\label{subcat}
The category {\bf CellRes} has many subcategories. Depending whether the restrictions are on the morphisms or objects, subcategories can allow us to study some subsets of cellular resolutions. We mention here few of the most commonly considered types of cellular resolutions.

Restricting the cellular resolutions to those coming from just labeled simplicial complexes gives a subcategory of {\bf CellRes} defined on simplicial complexes. This subcategory is full as all the morphisms are there. Every monomial ideal $I$ has a resolution in the category of cellular resolutions coming from simplicial complexes thanks to the Taylor resolution. 

Another possibility is to just look at the category of minimal cellular resolutions. We know that most of the constructions given in the following sections are not closed in this subcategory as they give non-minimal resolutions.

\section{Properties of the Category CellRes}
In this section we present some observations for the category CellRes. Among these are definition and results on homotopy on cellular resolutions, and forgetful functors to {\bf Top} and $\mathcal{C}_{\bullet}({\bf Mod}_S)$.

\begin{proposition}
The  cellular resolution 
$$0\leftarrow S\leftarrow0$$ supported on the empty complex is the initial object in the category {\bf CellRes}.
\end{proposition}
\begin{proof}
The empty complex is defined to have the label 1, and  has a cellular complex 
$$0\leftarrow S\leftarrow0$$ which is also the resolution of the ideal $(1)$.

From this resolution we have a map $({\bf f},f)$ to any cellular resolution $\mathcal{F}$ by taking the chain map ${\bf f}$ to be the zero map for $f_i$ when $i\geq 1$, and by defining $f_0$ to be the identity. The cellular map $f$ is  the embedding of the empty complex to the cell complex supporting the resolution $\mathcal{F}$. By definition of the initial object, the cellular resolution $0\leftarrow S\leftarrow0$ is an initial object in ${\bf CellRes}$.
\end{proof}

We have a well defined concept of homotopy for both cell complexes and chain complexes. The next definition lifts these definitions to {\bf CellRes}.

\begin{definition}
Let $({\bf f},f),({\bf g}, g): \mathcal{F}\rightarrow \mathcal{G}$ be two morphisms in {\bf CellRes}.
We say that $({\bf f},f)$ is homotopic to $({\bf g},g)$ if the components are homotopic, i.e. ${\bf f}\sim{\bf g}$ as chain maps and $f\sim g$ as continuous topological maps.
\end{definition}

Homotopies form a nice class of maps in {\bf Top}. Among the desirable properties is that they satisfy 2-out-of-3 property and 2-out-of-6 property, which are defined below. These also lift to cellular resolutions.

\begin{definition}
Let $K$ be a class of morphisms in a category $\mathcal{C}$. Two compasable morphism $f$ and $g$ are said to satisfy the \emph{2-out-of-3 property} if any two of the morphisms $f$, $g$, and $g\circ f$ are in $K$, then the third is too.

Three composable morphisms $f$,$g$, and $h$ satisfy the \emph{2-out-of-6 property} if $h\circ g$ and $g\circ f$ are in $K$, then so are $f$,$g$,$h$, and $h\circ g\circ h$.
\end{definition}

\begin{proposition}
Homotopy maps form a class of morphism that satisfy the 2-out-of-3 property and 2-out-of-6 property.

\end{proposition}
\begin{proof}
This follows from that in {\bf Top} and $\mathcal{C}_{\bullet}({\bf Mod}_S)$ category, these homotopy properties are satisfied. Hence we get that the homotopies in {\bf CellRes} also satisfy it as both components satisfy it.  \end{proof}

\begin{proposition}
{\bf CellRes} is a homotopical category.
\end{proposition}
\begin{proof}
Take the weak equivalences to be the homotopies defined above. Then this class of morphims contains identities and isomorphism. Moreover, the homotopy in both {\bf Top} and $\mathcal{C}_{\bullet}({\bf Mod}_S)$ satisfy the 2-out-of-3 property so, then does the homotopy in {\bf CellRes}.
\end{proof}

Another one of the properties {\bf CellRes} inherits from {\bf Top} is the enrichment by simplicial sets. Recall that the enrichment by a monoidal category was defined in the Section \ref{catsection}, Definition \ref{enrich}. 
\begin{proposition}
\label{enri}
The category {\bf CellRes} can be enriched with simplicial sets.
\end{proposition}
\begin{proof}
We want to show that for any pair of cellular resolutions $F$ and $G$ we can assign a simplicial set $sS(F,G)$. We know that the category {\bf Top} can be enriched with simplicial sets by taking for any pair $X$, $Y$ the simplicial set where the 0-simplices are the maps between $X$ and $Y$, the 1-simplices are the homotopies between the maps, and the higher simplices are the higher homotopies. 

We can defined the simplicial set in the same way in {\bf CellRes}. For any pair $F$ and $G$ take the 0-simplices to be the morhisms between them, the 1-simplices are the homotopies, and the higher homotopies are the higher simplices.

Then the above definition inherits the properties of enriched category from {\bf Top}.
\end{proof}

\begin{proposition}
The kernel (and cokernel) for the maps in $\bf{CellRes}$ do not exist in  general. 
\end{proposition}
\begin{proof}
Let $({\bf f},f):\mathcal{F}\rightarrow\mathcal{G}$ be  a morphism of cellular resolutions.
By definition, the kernel of a morphism is an object $\mathcal{K}$ and a map $({\bf k}, k):\mathcal{K}\rightarrow\mathcal{F}$, such that $({\bf f},f)\circ ({\bf k},k)=0$.

From the definition of composition of morphisms in ${\bf CellRes}$, we have that ${\bf f}\circ{\bf k}=0$ and $f\circ k=0$. In particular, this means that $\mathcal{K}$ is the kernel also as a chain complex, and as a requirement for the cellular resolutions it should be a free module. However, we know that the kernel of a module map of free modules is not necessarily free. Thus the kernel $\mathcal{K}$ may not even be a chain complex of free modules.

Similarly, cokernels do not always exist as the map on the level of modules does not give a free module $S^n$ as the cokernel at homological degree  in the chain complex in most cases. Again the examples with cokernel existing in the category $\bf{CellRes}$ are from cases with multiple connected components. \end{proof}

\begin{remark}
Because the category does not contain all the kernels and cokernels we know that it is not abelian. 
\end{remark}

We have two important forgetful functors from the category of cellular resolutions, one to chain complexes and one to topological spaces.

\begin{definition}
Let $\Phi: {\bf CellRes}\rightarrow\mathcal{C}_{\bullet}(\bf{Mod}_S)$ be a covariant functor taking a resolution 
$$\mathcal{F}_X:\ 0\leftarrow F_0\leftarrow F_1\leftarrow\ldots\leftarrow F_n \leftarrow\ldots\rightarrow0$$
to a chain complex $\mathcal{F}$ with the same $S$-modules and differential maps as in $\mathcal{F}_X$.
The morphism $({\bf f},f)$ between two cellular resolutions $\mathcal{F}_X$ and $\mathcal{F}_Y$ gets mapped to the chain map ${\bf f}$ under $\Phi$.
\end{definition}

\begin{definition}
 Let $\Psi:\bf{CellRes}\rightarrow\bf{Top}$ be a covariant functor, given by mapping $\mathcal{F}$ to the unlabeled cell complex $X$ supporting the cellular resolution. Then $\Psi$ maps a morphisms $({\bf f},f)$ to the cellular map $f$.
\end{definition}

\begin{proposition}
The forgetful functors $\Phi$ and $\Psi$  preserve weak equivalences, that is, they are homotopical.
\end{proposition}
\begin{proof}
This follows directly from the definition of the functors and the weak equivalences on {\bf CellRes}.\end{proof}

\section{Mapping cones and cylinders}
The mapping cone (and mapping cylinder) construction for chain complexes is modeled after the mapping cones (and cylinders) for topological spaces. This means that the two constructions are similar and one would expect both of them to work for cellular resolutions, which is indeed the case.
\subsection{ The mapping cone}
 Since we have that mapping cones are very similar in $\mathcal{C}_{\bullet}(\bf{Mod}_S)$ and {\bf Top}, we can use the definition of $\mathcal{C}_{\bullet}(\bf{Mod}_S)$ on cellular resolutions as well.
\begin{definition}
Let $\mathfrak{h}=({\bf h},h): \mathcal{F}\rightarrow \mathcal{G}$ be a morphism between two cellular resolutions. Then the mapping cone $C(\mathfrak{h})$ is the chain complex $C(\mathfrak{h})_i=\mathcal{G}_i\oplus \mathcal{F}_{i-1}$ for $i\geq2$ or $i=0$ with the differential $d_i(g,f)=({\bf h}(f)+d_i(g),-d_{i-1}(f))$, and $C(\mathfrak{h})_1=\mathcal{G}_1\oplus S(-\deg h_0({\bf 1}))$.

\end{definition}
\begin{remark}
Here $\mathcal{F}_i$ in the mapping cone is equal to $\mathcal{F}_i$ only in the case where we do not write down the grading for free modules in the resolution. If the grading is considered then we have a free module for each component of $\mathcal{F}_i$ with generators given by the differentials.
\end{remark}

\begin{proposition}
\label{rescone}
The mapping cone of cellular resolutions is in the category of {\bf CellRes}.
\end{proposition}

\begin{proof}
 Let $C(\mathfrak{h})$ be the mapping cone coming from $\mathfrak{h}\colon\mathcal{F}\rightarrow \mathcal{G}$. We have $C(\mathfrak{h})_i=\mathcal{G}_i\oplus \mathcal{F}_{i-1}$, which is a free module as both $\mathcal{G}_i$ and $\mathcal{F}_{i-1}$ are free modules. The differential $d_i(g,f)=({\bf h}(f)+d_i(g),-d_{i-1}(f))$ satisfies $\operatorname{ker} d_i\supseteq \operatorname{im} d_{i+1}$ as it is same as for chain complexes.  To show the other inclusion, we note that $(g,f)\in \operatorname{ker} d_i$ implies that $f$ is in the image of $d_i:\mathcal{F}_{i+1}\rightarrow \mathcal{F}_i$ and hence by the chain map property of commutative squares for {\bf h}, we have that ${\bf h}(f)$ is in the image for $d_i$ for $\mathcal{G}$. We can write ${\bf h}(f)=d_i(\psi(f'))$, and we get that $d_i({\bf h}(f'))+d_i(g)=0$. Then the fact $d_i$ is a differential in a cellular resolution implies that ${\bf h}(f')+g\in \operatorname{im} d_{i+1}$, so $g=d_{i+1}(g')-{\bf h}(f')$. Hence it is in the image of ${\bf h}+d_{i+1}$.
 
Let $X$ and $Y$ be the cell complexes of $\mathcal{F}$ and of $\mathcal{G}$, respectively. Consider the mapping cone for the associated cell complexes by the map $h$. It is the cell complex where we have all of $Y$, a single point coming from the complex $X$, and then a $(i+1)$-dimensional cell for each $i$ dimensional cell in $X$. Then the cellular complex $M$ for the mapping cone is 
$M_0= \mathcal{G}_0$, $M_1=\mathcal{G}_1\oplus S(-\deg h_0({\bf 1}))$ and higher degrees $M_i=\mathcal{G}_i\oplus \mathcal{F}_{i-1}$, up to the same abuse of notation as in the definition. So we se that this is indeed the same as the chain complex mapping cone, and so it has a cellular structure. 
\end{proof}
\begin{remark}
This mapping cone is in most cases not minimal, and can be very far from it. 
\end{remark}
\begin{remark}
If the map $\psi$ has the identity map $\psi_0:\mathcal{F}_0\rightarrow \mathcal{G}_0$, then the mapping cone will contain label 1. 
\end{remark}

\begin{proposition}
The mapping cone for cellular resolutions corresponds to the mapping cones in $\mathcal{C}_{\bullet}({\bf Mod}_S)$ and {\bf Top} via the forgetful functor.
\end{proposition}
\begin{proof}
This construction does match the topological construction for the mapping cone. This is because the cell complex associated to the mapping cone only adds one point for each connected component of $\mathcal{F}$ and  a number of other faces depending on the maps involved,  to the  cell complex of $\mathcal{G}$. Algebraically this can be seen from the fact that the free module in the homological degree 1 is $\mathcal{G}_1\oplus \mathcal{F}_0$, where generators correspond to the points in the cell complex. \end{proof}

\begin{example}
\label{conex}
Let us consider the cellular resolutions in Figure \ref{coneex}, with  $\mathcal{F}$ the cellular resolution
$$0\leftarrow S\leftarrow S^3\leftarrow S^3\leftarrow S\leftarrow0$$
coming from the complex $X$ in the Figure \ref{coneex},
and $\mathcal{G}$ 
$$0\leftarrow S\leftarrow S^3\leftarrow S^2\leftarrow 0$$ from the complex $Y$ in the same figure. Both have the ideal $(ab,bc,cd)\subset S=k[a,b,c,d]$ as their label ideal. The map between the two is identity, and it embeds the minimal resolution to the non-minimal one. Then using the definition of mapping cones we get the resolution 
$$0\leftarrow S^3\oplus S\leftarrow S^3\oplus S^3\leftarrow S\oplus S^2\leftarrow 0,$$
with the maps shown in the Figure \ref{coneres}.
\begin{figure}
\begin{center}
\includegraphics[scale=1]{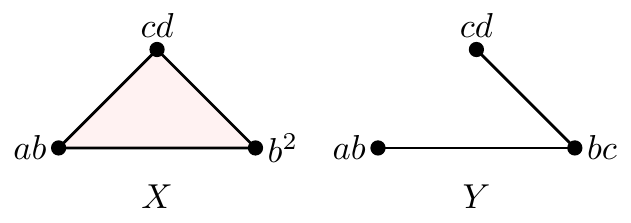}
\end{center}
\caption{Cell complexes of the cellular resolutions in Example \ref{conex}.}
\label{coneex}
\end{figure}
\begin{figure}
\begin{center}
\includegraphics[scale=1]{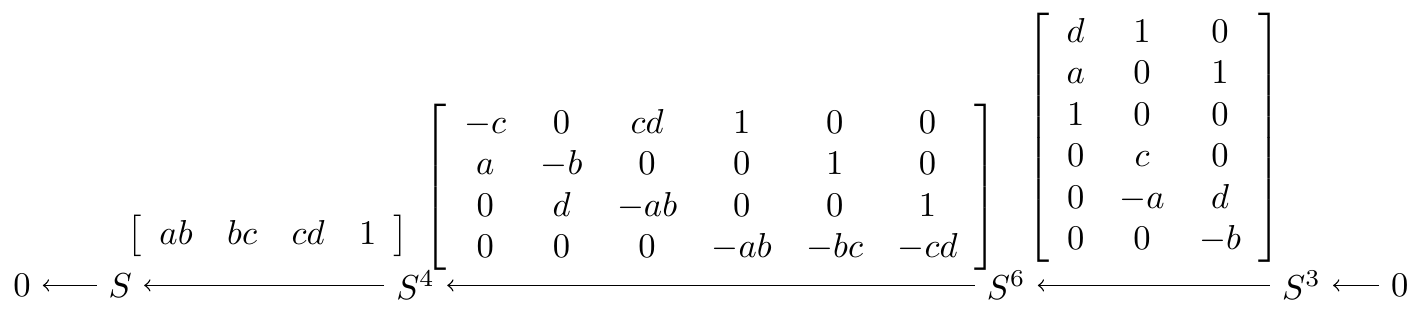}
\end{center}
\caption{Cellular resolution of Example \ref{conex}.}
\label{coneres}
\end{figure}
The labeled cell complex associated to this is then the cell complex in Figure \ref{cone}.
\begin{figure}

\begin{center}
\includegraphics[scale=1]{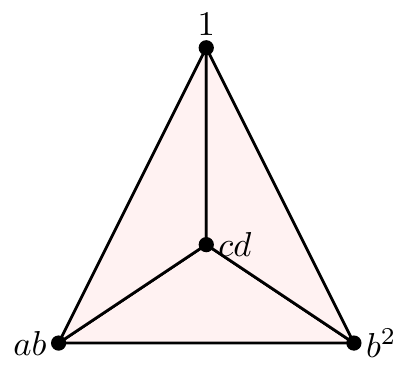}
\end{center}
\caption{The mapping cone of $X$ and $Y$.}
\label{cone}
\end{figure}
We can see that this cell complex is the same as the topological mapping cone of the embedding of $X$ to $Y$.
\end{example}

This "adding a point" property can then be used to construct cellular resolutions from known ones. For the next result we consider cellular resolutions with connected cell complex.

\begin{corollary}
Let $\mathcal{F}$ be a cellular resolution for the module $S/I$.
If a map $\mathfrak{f}:\mathcal{F}\rightarrow \mathcal{G}$ has $f_0$ being a multiplication by a monomial $m$, then the mapping cone $C(\mathfrak{f})$ will give the resolution of $S/(I,m)$.
\end{corollary}
\begin{proof}
Let $\mathfrak{f}:\mathcal{F}\rightarrow \mathcal{G}$ be a morphisms such that $f_0$ is a multiplication by a monomial $m$.
We know that the mapping cone is a cellular resolution by Proposition \ref{rescone}. Then the label on the "new" vertex of the mapping cone is given by the map $h_0$ acting on the element 1.  Then by definition of $h_0$, we get $h_0(1)=m$. 
Thus the labels on the cell complex of the mapping cone are $I\cup \{m\}$. Then by the Proposition \ref{resmodule}, we have that  the mapping cone $C(\mathfrak{f})$ is the resolution of $S/(I,m)$. 
\end{proof}

In particular, if one can construct a new morphism to the mapping cone with suitable monomial multiplication, iterating the above process can give specific cellular resolutions. However, in general finding these components proves challenging.

This kind of iterative behaviour of mapping cones was studied by Herzog and Taniyama in \cite{HT} for resolutions to construct minimal resolution in a purely algebraic setting. Later on Dochtermann and Mohammadi showed in \cite{DM} that the
 minimal free resolutions  from iterated mapping cones of \cite{HT} are cellular resolutions.
 
For completeness we state the results from \cite{HT} and \cite{DM}.
 
\begin{definition}
A monomial ideal $I\in S$ with a minimal set of generators $G(I)$ and an order $u_1,u_2,\ldots,u_m$ on the generators is said to have linear quotients if the colon ideal $(u_1,u_2,\ldots,u_{j-1}):u_j$ is generated by some subset of the variables $\{x_1,x_2,\ldots,x_n\}$ of $S$ for all $1\leq j\leq m$.

We define the set
$$set(u_j)=\{k\in[n]:x_k\in(u_1,u_2,\ldots,u_{j-1}):u_j\} \textrm{ for }j=1,\ldots,m.$$
\end{definition}

\begin{proposition}[\cite{HT}]
Let $I$ be an ideal with linear quotient with respect to the ordering $u_1,u_2,\ldots,u_m$, and let $I_j=(u_1,u_2,\ldots,u_j)$ and $L_j=I_j:u_{j+1}$.
We have the exact sequence
$$0\rightarrow R/L_j\rightarrow R/I_j\rightarrow R/I_{j+1}\rightarrow0.$$
The map $R/L_{j+1}\rightarrow R/I_j$ is  multiplication by $u_{j+1}$. Let $F^j$ denote the graded free resolution of $R/I_j$ and $K^j$ the Koszul complex of the regular sequence $x_{k_1},x_{k_2},\ldots,x_{k_l}$ with $k_i\in set(u_{j+1})$, and let $\psi^j:K^j\rightarrow F^j$ be a graded chain map lifting $\psi:R/L_j\rightarrow R/I_j$. Then the mapping cone $C(\psi^j)$ gives the free resolution of $R/I_{j+1}$. Iterating the process we get a graded free resolution of $R/I$. 
\end{proposition}

\begin{theorem}[\cite{DM}, Theorem 3.10]
If the ideal $I$ has linear quotients with respect to some order $u_1,u_2,\ldots,u_m$ and that the decomposition function is regular, then the minimum resolution obtained via iterated mapping cones is cellular and it is supported on regular CW-complex.
\end{theorem}

\subsection{ The mapping cylinder}
\label{cylinder}
One can also construct the mapping cylinder of two cellular resolutions. Just like the mapping cone construction, it agrees with both the topological mapping cylinder and the chain complex mapping cylinder.

\begin{definition}
Let $\mathcal{F}$ and $\mathcal{G}$ be two cellular resolutions with labeled connected cell complexes $X$ and $Y$, and suppose there is a map $\mathfrak{h}=({\bf h},h')\colon \mathcal{F}\rightarrow \mathcal{G}$. Then the \emph{mapping cylinder}, $D(\mathfrak{h})$, is the cellular resolution given by the following data:
we set $D_0=S$ as the topological mapping cylinder is connected, and the other free modules are given by
$$D_1=\mathcal{F}_1\oplus \mathcal{G}_1 \textrm{ and }D_i=\mathcal{F}_i\oplus \mathcal{G}_i\oplus \mathcal{F}_{i-1}\textrm{ for }i\geq 2.$$
The differentials of the mapping cylinder are then given by
$$d_1(f,g)=(d_{F1}(f)+d_{G1}(g)), d_2(f,g,f')=(d_{F2}(f)-id(f'),d_{G2}(g)+h(f'))$$
$$\textrm{ and }d_i(f,g,f')=(d_{Fi}(f)+id(f'),d_{Gi}(g)-h(f'),d_{Fi-1}(f')).$$
\end{definition}

\begin{proposition}
\label{mapcyl}
The mapping cylinder is a cellular resolution.
\end{proposition}
\begin{proof}
Taking the topological mapping cone of the labeled complexes $X$ and $Y$ along $\psi'$, and then computing the cellular chain complex for the mapping cone gives us the same chain complex as the above definition. Therefore we know that the mapping cylinder does have a cellular structure.

Next we need to show that the mapping cylinder is indeed a resolution, i.e. that it is acyclic. 
The kernel of $d_i$ is $\ker d_{F_i}\oplus \ker d_{G_i} \oplus (\ker \operatorname{id}\cup \ker\psi\ \cup\ \ker d_{F_{i-1}})=\ker d_{F_i}\oplus \ker d_{G_i} \oplus F_{i-1}$. It is not hard to see that $\ker d_i\subseteq \textrm{im } d_{i+1}$, and so the chain complex given by the mapping cone is acyclic if the two complexes are acyclic.

To show that this resolution is supported on some cell complex we again consider the case of the associated cell complexes. 
Take the mapping cylinder for $X$ and $Y$. Following the definition it is a cell complex with $X$ and $Y$ as subspaces, and for each cell $x$ of dimension $i$ in $X$, a $(i+1)$-dimensional cell $m_x$ in the mapping cylinder. The label of $m_x$ is the least common multiple of the labels of $x$ and $f(x)$.
Then we can construct the cellular complex of the mapping cone $M$ and get
$M_0=\mathcal{G}_0$, $M_1=\mathcal{G}_1\oplus \mathcal{F}_1$, and $M_i=\mathcal{G}_i\oplus \mathcal{F}_i\oplus \mathcal{F}_{i-1}$.
Thus we see that it is the same as the mapping cone defined for cellular resolutions, and hence there is a cellular structure.
\end{proof}

\begin{example}
\label{cylinderex}
Let us consider the same map and cellular resolutions as in the Example \ref{conex} (of a mapping cone). 
For the mapping cylinder we get the resolution shown in the Figure \ref{cylres}.

\begin{figure}
\begin{center}
\includegraphics[scale=1]{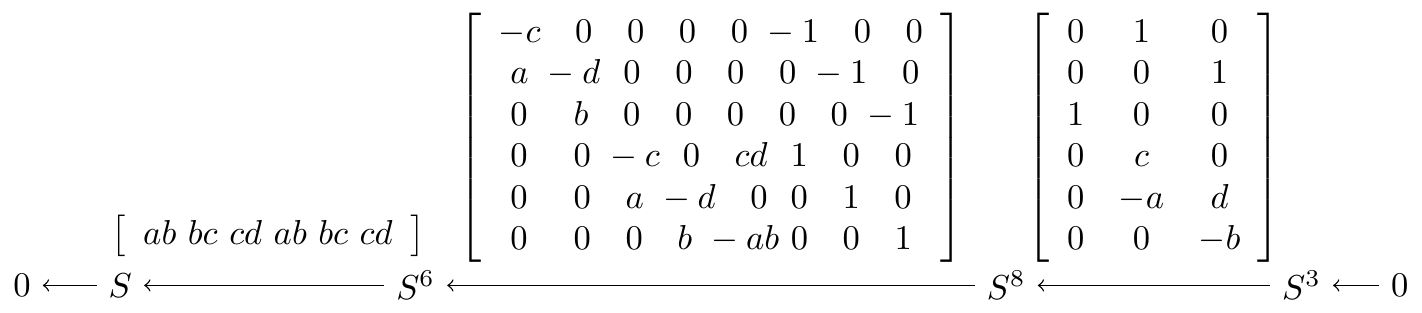}
\end{center}
\caption{The cellular resolution of the mapping cylinder in Example \ref{cylinderex}.}
\label{cylres}
\end{figure}
\end{example}

Since a single mapping cone is a cellular resolution, one can ask whether multiple mapping cones can be glued together to form cellular resolutions. The answer to that is yes, and the result is in Proposition \ref{glumap} below.

\begin{lemma}
\label{glu}
Let $\mathcal{F}$ and $\mathcal{G}$ be cellular resolutions, such that both contain the sub-resolution $\mathcal{H}$. Then gluing $\mathcal{F}$ and $\mathcal{G}$ together along $\mathcal{H}$, by identifying the $\mathcal{H}$ in $\mathcal{F}$ with the $\mathcal{H}$ in $\mathcal{G}$, gives a cellular resolution.
\end{lemma}
\begin{proof}
Let $\mathcal{H}$ be a sub-resolution of both $\mathcal{F}$ and $\mathcal{G}$. Then we can write $\mathcal{F}_i=\mathcal{F}'_i\oplus \mathcal{H}_i$ and $\mathcal{G}_i=\mathcal{G}'_i\oplus \mathcal{H}_i$ for $i\geq1$. The differentials in $\mathcal{F}$ can be written as $d(f,h)=(d(f)+d_F(h), d(f)+d_{\mathcal{H}}(h))$, and similarly for $\mathcal{G}$ $d(g,h)=(d(g)+d_{\mathcal{G}}(h),d(g)+d_{\mathcal{H}}(h))$. Let $\mathcal{F}\mathcal{G}$ be the glued cellular resolution. Identifying the two resolutions along $\mathcal{H}$ gives $\mathcal{F}\mathcal{G}_i=\mathcal{H}_i\oplus \mathcal{F}'_i\oplus \mathcal{G}'_i$ for $\geq1$ and $F\mathcal{G}_0=S^{\max(n_{\mathcal{F}},n_{\mathcal{G}})}$ where $\mathcal{F}_0=S^{n_{\mathcal{F}}}$ and $\mathcal{G}_0=S^{n_{\mathcal{G}}}$. The $S$-module $\mathcal{F}\mathcal{G}_i$ is a direct sum of free modules with differentials  $d(h,f,g)=(d(f)+d(g)+d_{\mathcal{H}}(h), d(f)+d_{\mathcal{F}}(h),d(g)+d_{\mathcal{G}}(h))$. The differentials are made of sums of acyclic differentials, hence they also give an acyclic chain complex.

The above shows that $\mathcal{F}\mathcal{G}$ is a resolution. Now we want to show that it is supported on a cell complex. We note that the resolution has everything corresponding to the cell complex of $\mathcal{F}$ and to the cell complex of $\mathcal{G}$ as well. They are connected along $\mathcal{H}$. Taking  the cell complex obtained by gluing the associated complexes of $\mathcal{F}$ and $\mathcal{G}$ along the cell complex of $\mathcal{H}$, gives us a cell complex that has $\mathcal{F}\mathcal{G}$ as its cellular complex.
\end{proof}

\begin{proposition}
\label{glumap}
Let $\mathcal{D}$ be a finite diagram of cellular resolutions.
Then gluing mapping cylinders into $\mathcal{D}$, gives a new cellular resolution.
\end{proposition}
\begin{proof}
Let $\mathcal{D}$ be a finite diagram of cellular resolutions. Then for each morphism $\mathfrak{f}_{ij}=({\bf f},f):D^i\rightarrow D^j$ in it, we can construct the mapping cylinder. By Proposition \ref{mapcyl} the mapping cylinders are cellular resolutions. 
In the mapping cylinder of $\mathfrak{f}_{ij}$ both $D^i$ and $D^j$ are sub-resolutions. 
Lemma \ref{glu} tells us that any two mapping cylinders glued along $D^i$ or $D^j$ are a cellular resolution. Thus by gluing the mapping cylinders together along the common components of $\mathcal{D}$ one at a time gives us a cellular resolutions, while the components are not connected.

In the case we have a cellular resolution $\mathcal{F}$, obtained by gluing mapping cylinders, that contains two (or more) copies of a cellular resolution $D^i$ we can glue the cellular resolution to itself along the sub-resolution $D^i$. This corresponds to the situation where we have more than one map between two objects in the diagram, for example a composition of maps being equal to another map.  Because fo this one can view the resolution before gluing consisting of the resolutions $D^i$, $D^j$, and two disjoint mapping cones $M^1$ and $M^2$ between them. The piece of the resolution 
$\mathcal{F}_k$ can be written as $D^j_k\oplus M^1_k\oplus M^2_k\oplus D^i_k\oplus D^i_k$. 
Trying to use the approach with the differentials as in Lemma \ref{glu}, we get a chain complex $\tilde{\mathcal{F} }$ which has $\tilde{\mathcal{F} }_k=D^j_k\oplus M^1_k\oplus M^2_k\oplus D^i_k$ with a differential $d(f,a)=(d(f),d_1(f)+d(a)+d_2(f))$. This chain complex is not acyclic. The differentials will not cover all the elements after homological degree 2. The missing elements in the kernel of the maps come from that after gluing, both mapping cylinders in degree $k+1$ will have elements mapping to the same element of the image of $D^j$ in $D^i_k$.
To make the glued resolution acyclic, we want to identify the mapping cylinders corresponding to a same map. This is done by adding a free module $G$ in homological degree $n+2$ for each of the generators in homological degree $n$ of $D^j$ such that $d(G)$ maps exactly to the degree $n+1$ modules coming from the mapping cylinder component of a single generator. The added modules provide the needed elements to the maps to cover the kernel elements coming from the mapping cylinders. Hence we get an acyclic resolution after the glueing and identifying the mapping cylinders. 

On the supported cell complex the gluing without identifying the mapping cones corresponds to having a hole in the complex. Adding the extra pieces is equivalent to adding in an $(n+2)$-cell for each $n$-cell in the complex of $D^j$ such that the $(n+1)$-cells in the mapping cylinder corresponding to an $n$-cell form the boundary for the $(n+2)$-cell. 

So now we have that gluing the mapping cylinders together one common component at a time gives a cellular resolution at each step.
With the assumption that our diagram is finite, we have that eventually each common component of the mapping cylinders has been glued, and we have a cellular resolution.
\end{proof}

\begin{definition}
\label{gluedef}
Let $D$ be a diagram of labeled cell complexes. Then we take their gluing along the maps to be the same topological cell complex as with the usual gluing of complexes and the labels on the vertices to be the least common multiple of all the labels of the vertices that are glued together.
\end{definition}

\begin{proposition}
\label{gluemap}
Let $({\bf f}, f): \mathcal{F}\rightarrow \mathcal{G}$ be a morphism of cellular resolutions. Then gluing $\mathcal{F}$ to $\mathcal{G}$ along the chain map  {\bf f} of the morphism gives a cellular resolution.
\end{proposition}
\begin{proof}
Let $({\bf f}, f): \mathcal{F}\rightarrow \mathcal{G}$ be a morphism of cellular resolutions. Let us consider the gluing along {\bf f } given by $\mathcal{F}\sqcup \mathcal{G}/\sim$, where $x\in \mathcal{F}_i\sim x'\in \mathcal{G}_i$ if we have $f_i(x)=x'$.
Firstly we show that $\mathcal{F}_i\sqcup \mathcal{G}_i/\sim$ is a free module. The map $f_i$ maps each generator of $\mathcal{F}_i$ to an element $x'$ of $\mathcal{G}$ or to 0. This means that each element can be written as a combination of the basis elements of $\mathcal{G}$, so $\mathcal{F}_i\sqcup \mathcal{G}_i/\sim\cong \mathcal{G}_i$ is a free module. Also this gives that the differentials are just the same as in $\mathcal{G}_i$. Furthermore, as $\mathcal{G}$ is a cellular resolution, $\mathcal{F}\sqcup \mathcal{G}/\sim$ also has a labeled cell complex and is a cellular resolution. 
\end{proof}

\section{Products, coproducts, and the tensor product}
\subsection{The product}
\label{prod}

We know that the categories {\bf Top} and $\mathcal{C}_{\bullet}(\bf{Mod}_S)$ have different types of products. In $\mathcal{C}_{\bullet}(\bf{Mod}_S)$ it is (in a finite case) the same as the direct sum, whereas in {\bf Top} we get a connected cell complex instead of the disjoint union. This tells us that we cannot use either of the known definitions for the product in {\bf CellRes}, as this would then either not preserve the topological structure or not preserve the algebraic structure of a product. 

We can, however, lift the construction of the topological product with trivially labeled cell complex to {\bf CellRes}.

\begin{definition}
\label{prodef}
Let $\mathcal{F}$ be a cellular resolution supported on the cell complex $X$, with $F_i=S^{a_i}$ and differentials $\partial_i$. Let $T_n$ be the cellular resolution coming from the $n$-simplex with labels 1.

The product of a cellular resolution $\mathcal{F}$ with $T_n$, is the cellular resolution $P$, with $P_0=S^{a_0}$ and 
$$P_i=S^{\sum_{k=0}^{i-1}\binom{n}{k+1}a_{i-k}}\ \mathrm{ for }\ i\geq1.$$
The differential $d_i:P_{i+1}\rightarrow P_i$ of the product in the $i\times i+1$ matrix form is given by
$$d_i=\left[\begin{array}{cccccc}
[\partial_{i+1}]_n&[\operatorname{id}]_{n\binom{n}{2}} &0&0&\cdots&0\\
0&[\partial_i]_{\binom{n}{2}}& [\operatorname{id}]_{\binom{n}{2}\binom{n}{3}}&0&\cdots&0\\
0&0&[\partial_{i-1}]_{\binom{n}{3}}&[\operatorname{id}]_{\binom{n}{3}\binom{n}{4}} &\cdots&0\\
\vdots&\vdots& & & & \vdots\\
0&0&\cdots&0&[\partial_1]_{\binom{n}{i}}&[\operatorname{id}]_{\binom{n}{i}\binom{n}{i+1}}\\
\end{array}\right],$$
where $[\partial_i]_n$ denotes the $n\times n$ diagonal matrix where each diagonal entry is $\partial_i$, and $[\operatorname{id}]_{\binom{n}{i}\binom{n}{i+1}}$ denotes the matrix of the differential of the simplex with identity map entries. 

The cell complex supporting the product is the topological product of $X$ and the $n$-simplex. The orientation of the product complex is given by each copy of $X$ having the same orientation as $X$ and each $n$-simplex having the same orientation and the new faces being ordered by the order of the one dimension lower faces of $X$. 

The projection from $P$ to $\mathcal{F}$ is a pair $({\bf p}_1,p_1')$ where $p_1'$ is the standard topological projection and ${\bf p}_1$ is a compatible chain map with $p_{1,0}=\operatorname{id}$.
The projection from $P$ to $T_n$ is a compatible pair $({\bf p}_2,p_2')$ where $p_2'$ is also the standard topological projection, and ${\bf p}_2$ is a compatible chain map sending all labels to 1. 
\end{definition}

\begin{remark}
The above is not a product in the usual sense, as the map in the universal property is not unique.  However, the maps are unique up to homotopy and this suffices for our purposes.
\end{remark}

\begin{example}
\label{prod1}
Let $S=k[x,y]$.
Let $\mathcal{F}$ be the cellular resolution associated to the Koszul complex of $(x,y)$. Then consider the product with $T_2$. Using Definition \ref{prodef}, we get that the resolution is
$$0\leftarrow S\xleftarrow{d_1}  S^4\xleftarrow{d_2}  S^4\xleftarrow{d_3}  S\leftarrow 0$$ with maps
$d_1=[x\ y\ x\ y]$,$d_2=\left[\begin{array}{cccc}
-y&0&-1&0\\
x&0&0&-1\\
0&-y&1&0\\
0&x&0&1
\end{array}\right]$ and $d_3=\left[\begin{array}{c}
1\\ -1\\ -y\\ x
\end{array}\right]$.
By the definition of a product, if we have something mapping to the two components, we should have a map to the product, too, such that the maps commute. 
We can take $\mathcal{F}$ as the cellular resolution mapping to itself and to $T_2$. Then we have a map $({\bf f},f)$ from $\mathcal{F}$ to the product. On the level of cell complexes, we can draw a picture of the product diagram, which is shown in Figure \ref{prodex}.
 \begin{figure}

\begin{center}
\includegraphics[scale=1]{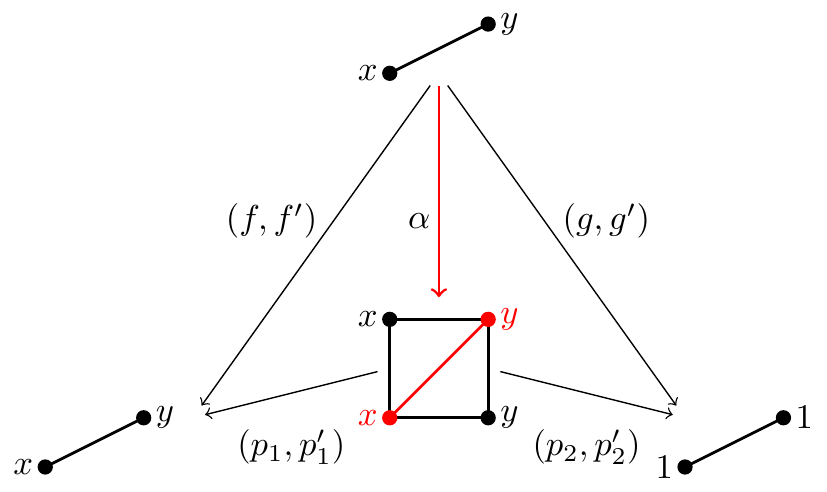}
\end{center}
 \caption{The product diagram on the cell complex level of Example \ref{prod1}.}
  \label{prodex}
 \end{figure}
 The red part of the diagram marks where a continuous map $\alpha$ would map to in a product, if we consider the topological product for the cell complexes. Clearly, $\alpha$ is not  a cellular map, as it maps the edge of $\mathcal{F}$ to a higher dimensional face. 
 So one can not take the topological map $\alpha$ as a component for the morphism making the diagrams commute in ${\bf CellRes}$. We know that the vertices have to be mapped as in the drawn red map to make the diagram commute, so we can choose a cellular map $\beta$ that maps the vertices in this way. This leaves two different options how we map the edge, up or down from the red line, and with just the requirement to make the diagram commute, there is no clear choice between the two options. 
 
 This non-uniqueness can also be seen algebraically. We refer to the left and the right side of the diagram in the Figure \ref{prodex}. To make the left side diagram commute on with $p_{1,1}=\operatorname{id}$ and $f_1=\operatorname{id}$, we see that $\beta_1=\operatorname{id}$.  Then also using that the right diagram must commute we get $\beta_1=\left[\begin{array}{cc}
 1&0\\ 0&0\\ 0&0\\ 0&1
 \end{array}\right]$. Finally to choose the map $\beta_2$, we run into the situation where both $\left[\begin{array}{c}
 1\\ 0\\ 0\\ x
 \end{array}\right]$ and $\left[\begin{array}{c}
 0\\ 1\\ y\\ 0
 \end{array}\right]$ make the diagram commute. 
 Hence the product defined in Definition \ref{prodef} does not give a product in the category-theoretic sense, because the map in the universal property is not unique. However, the choices are homotopic to each other, so the map is unique up to homotopy. 
\end{example}

These choices in choosing the map to the product $P$ arise at level where the map of the topological product map subdivides a cell. 
\begin{proposition}
Let $P$ be the product of cellular resolutions $\mathcal{F}$ and $T_n$. Let $\mathcal{Z}$ be a cellular resolution mapping to both $F$ and $T_n$. Then the vertex map from the cell complex supporting $\mathcal{Z}$ to the cell complex of $P$ is well-defined and compatible with a module map. 
\end{proposition}
\begin{proof}
Let $X_{\mathcal{Z}}$ be the cell complex supporting $\mathcal{Z}$, and let $X_{P}$ be the cell complex supporting $P$.
Let $\alpha:X_{\mathcal{Z}}\rightarrow X_P$ be the continuous map between cell complexes that makes the topological product diagram commute.

In the dimensions where $\alpha$ does not subdivide cells it satisfies the conditions of a cell map. 
Since the other maps in the commutative diagram are cellular, we know that they map vertices to vertices. Then commutativity implies that $\alpha$  must map vertices to vertices. 
Let us denote by $\beta$ the cellular part of $\alpha$.

Similarly, if we just consider the chain map part ${\bf \beta}$, we can compute $\beta_0$ due to commutativity of the triangle between $\mathcal{Z}$, $P$ and $\mathcal{F}$. The map $\beta_0$ is compatible with the cellular map $\beta$, which again follows from the commutativity of the diagram.
\end{proof}

\begin{remark} 
\label{subdiv}
We note two things about the nature of the subdivision in the product $P$.
\begin{itemize}
\item[(1)] In the product, the only subdivided cells are the ones not in either one of the components of the product.
\item[(2)]Subdivision maps a face of $Z$ of dimension $d$ to a  face in $X_P$ of dimension $d+1$, and the higher dimensional cell is divided into two parts.
\end{itemize}
\end{remark}

\begin{proposition}
\label{prodapprox}
Let $P$ be the product of the cellular resolutions $\mathcal{F}$ and $T_n$. Let $\mathcal{Z}$ be a cellular resolution mapping to both $\mathcal{F}$ and $T_n$. Denote the cell complexes the cellular resolutions are supported on $X_P$, $X_{\mathcal{F}}$, and $X_{\mathcal{Z}}$.

Let $\alpha:X_{\mathcal{Z}}\rightarrow X_P$ be the continuous map that makes the topological product diagram commute.
Then there exists a cellular map $\beta$ that is homotopic to the unique topological map $\alpha$ in the topological product.

Moreover, the cellular map $\beta$ together with a compatible chain map forms a morphism that gives commutative product diagram in {\bf CellRes}.
\end{proposition}
\begin{proof}
From the cellular approximation theorem we know that if we have a continuous map between two CW-complexes, then there exists a cellular map that is homotopic to the continuous map. In our case take the continuous map to be from $\mathcal{Z}$ to $P$ in the product, then there exists a cellular map $\alpha:\mathcal{Z}\rightarrow P$. 
We know that the cellular approximation $\beta$ to the unique topological map $\alpha$ is equivalent to $\alpha$ up to the first subdivision of cells in $P$. 
Let $i\geq 2$ be the dimension of the faces where we get the first subdivision, and let $\mathcal{F}\in \mathcal{Z}$ be one such face. From our observations in Remark \ref{subdiv} we have that $\alpha(\mathcal{F})$ is contained in some cell $\mathcal{G}$ of dimension $i+1$, and the projection of that cell is $f(\mathcal{F})$. Since $\mathcal{G}$ is of dimension $i+1$ and it is "purely a product face", then the boundary $\mathcal{G}$ also maps to $f(\mathcal{F})$ under the projection. Moreover, the boundary of $\mathcal{F}$ gets mapped to the boundary of $\mathcal{G}$ by continuity, and $\alpha(\mathcal{F})$ divides $\mathcal{G}$ into two parts. Combining these observations we can choose the boundary (with only entire faces chosen) of one of the halves of $\mathcal{G}$ to be $\beta(\mathcal{F})$. Then $\beta$ commutes with the other maps in this dimension. Furthermore as $\beta$ is a cellular map, the higher dimensional cells will also satisfy the commutativity requirements due to mapping in the same way as the $i$-dimensional one.
On the algebraic side we can then construct the algebraic map based on $\beta$.
\end{proof}

\begin{proposition}
The product construction gives a product up to homotopy, that is, $P$ is a cellular resolution, there is a map to each component of the product from $P$ and it satisfies the universal property up to homotopic maps. 
\end{proposition}
\begin{proof}
To show that the product is a cellular resolution, we only need to show that it is an acyclic chain complex as it is the cellular complex of a labelled cell complex. 
A simple computation on the defined differentials shows that $d_i\circ d_{i+1}=0$, so the $\operatorname{im} d_{i+1}\subset \operatorname{ker} d_i$. 
Let us consider the kernel of the map $d_i$. We know that the kernel of each component of $d_i$ is contained in the image as they are from cellular resolutions, and so the whole kernel is. Thus we have that the product construction with $T_n$ gives a cellular resolution. 

A product must also satisfy the universal property, so let us consider the cellular resolution $\mathcal{Z}$ with the property that
$\mathcal{Z}$ maps to both $\mathcal{F}$ and $T_n$. As the product $P$ has the same cell complex as the topological product, and the projection maps associated to it are also same for the topological products, we know that we have a cellular continuous map $h'$ from the cell complex of $\mathcal{Z}$ to the product $P$ by Proposition \ref{prodapprox}. 
\begin{figure}
\begin{center}
\includegraphics[scale=1]{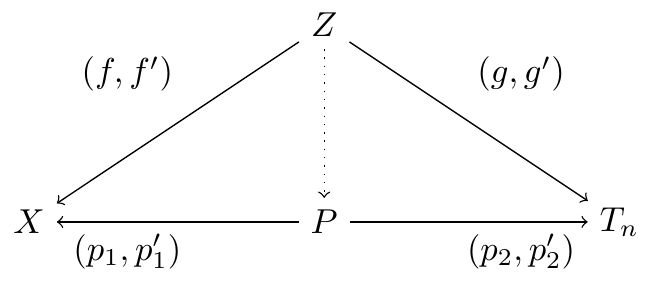}
\end{center}
\caption{A product diagram in {\bf CellRes}.}
\label{proddiag}
\end{figure}
By the same proposition we have that the diagram in the Figure \ref{proddiag} commutes for any of the cellular approximations and a compatible chain map.  Hence we get that the universal property holds up to homotopic choice of a {\bf CellRes} morphism. \end{proof}

The construction used in the product can be applied to any two cellular resolutions. However the resulting cellular resolution may not satisfy any of the product properties. In the case that the two cellular resolutions share labels with greatest common divisor other than 1, we do not even get well defined maps from the product construction to the components. 
\begin{proposition}
Let $H$ be the set of labels in the product construction $P$ for two cellular resolutions $\mathcal{F}$ and $\mathcal{G}$ with label sets $I$ and $J$, respectively.  If the maps from $H$ to $I$ and $J$ are generated by maps coming from the topological product of the cell complexes of $\mathcal{F}$ and $\mathcal{G}$ and they are compatible with chain maps between the resolutions, then  we get that $P$ is a product in {\bf CellRes}.
\end{proposition}
\begin{proof}
If the ideals associated to $\mathcal{F}$ and $\mathcal{G}$ do not have independent generator sets from each other, that is the generators between the sets have no non-trivial common divisors, then the label maps cannot be compatible. This follows from that we cannot construct a map from the "product" to $\mathcal{F}$ and $\mathcal{G}$.

 Let us assume that the genertor sets $I_{\mathcal{F}}$ and $I_{\mathcal{G}}$ are independent of each other. Then the associated cell complex is still the topological product, so we can take the morphisms as in the $T_n$ case but the distinction that the projection from $P$ to $\mathcal{F}$ maps all labels from $J$ to 1, and labels from $I$ to 1 with the other projection. 

Then we can apply the same arguments as in the $T_n$ product case to construct the cellular resolution using the cellular approximation to the topological product map. 
\end{proof}
\begin{remark}
The above proposition gives that we have the product with any cellular resolution with  cell complex having labels 1. 
\end{remark}

\subsection{The coproduct}
Next we move on the coproduct. Unlike with the product, both in {\bf Top} and $\mathcal{C}_{\bullet}(\bf{Mod}_S)$ the (finite) coproduct is a disjoint union. Thus the construction can be lifted to celluar resolutions directly.

\begin{definition}
\label{coprod}
The coproduct, $\mathcal{F}\sqcup \mathcal{G}$, of two cellular resolutions $\mathcal{F}$ and $\mathcal{G}$ is a direct sum of the cellular resolutions, so we have $(\mathcal{F}\sqcup \mathcal{G})_i=\mathcal{F}_i\oplus \mathcal{G}_i$. The labeled cell complex supporting the coproduct resolution is the disjoint union of the two labeled cell complexes, which also is the coproduct of cell complexes. 
\end{definition}

\begin{proposition}
The coproduct defined in Definition \ref{coprod} is a cellular resolution and satisfies the definition of category-theoretical coproduct.
\end{proposition}
\begin{proof}
Let $\mathcal{F}\sqcup\mathcal{ G}$ be the coproduct of $\mathcal{F}$ and $\mathcal{G}$ as defined above. 
The direct sum of two cellular resolutions is still a cellular resolution, as direct sums of chain complexes preserve exactness. The differentials are just the maps for the disjoint cell complex.
The maps from $\mathcal{F}$ and $\mathcal{G}$ to $\mathcal{F}\sqcup \mathcal{G}$ are embeddings of the cellular resolutions. 

Lastly, for this to be a coproduct we need the universal property. Let $\mathcal{Z}$ be a cellular resolution such that both $\mathcal{F}$ and $\mathcal{G}$ map to it. Since Definition \ref{coprod} is the same for chain complexes and topological spaces, we have a unique topological map and a unique chain map from $\mathcal{F}\sqcup \mathcal{G}$ to $\mathcal{Z}$. To show that these two maps are compatible with each other and form a unique {\bf CellRes} morphism consider the diagram of the coproduct. We have commutativity so the maps $\mathfrak{a}=({\bf a},a)\colon \mathcal{F}\rightarrow\mathcal{ F}\sqcup \mathcal{G}$, $\mathfrak{b}=({\bf b},b)\colon\mathcal{F}\rightarrow \mathcal{Z}$ and $({\bf e}, e)\colon\mathcal{F}\sqcup \mathcal{G}\rightarrow\mathcal{Z}$ satisfy ${\bf a}={\bf e}\circ{\bf b}$ and $ a= e\circ b$.
Let $x$ be a cell in the cell complex of $\mathcal{F}$ and $e_x$ the generator associated to it. From the cellular resolution maps $\mathfrak{a}$ and $\mathfrak{b}$ we know that ${\bf a}(e_x)$ corresponds to the cell $a(x)$, and ${\bf b}(e_x)$ corresponds to the cell $b(x)$. Commutativity tells us that $({\bf e}, e)$ will satisfy compatibility for all elements in $\mathcal{F}\sqcup \mathcal{G}$ coming from $\mathcal{F}$. Since the arguments also hold for $\mathcal{G}$, we get that $\mathfrak{e}=({\bf e},e)$ is a cellular resolution map.\end{proof}

\begin{proposition}
The category {\bf CellRes} has all finite coproducts.
\end{proposition}
\begin{proof}
We have that the coproduct of any two cellular resolutions exists. Then one can compute the coproduct of finitely many cellular resolutions by taking the coproduct inductively. At each step this is still the coproduct of two cellular resolutions, and so we have that each finite coproduct is in {\bf CellRes}.
\end{proof}

\begin{remark}
We only consider finite cellular resolutions, so an infinite coproduct would produce an infinite cellular resolution and hence we do not have infinite coproducts.
\end{remark}

\begin{definition}
Let $I$ be a set, and $\mathcal{F}$ a cellular resolution. Then the repeated coproduct over $I$, $\sqcup_{i\in I}\mathcal{F}$, is called the \emph{copower} (over $I$) and denoted by $I\odot \mathcal{F}$
such that the morphisms satisfy $Hom(I\odot \mathcal{F}, \mathcal{G})\cong Hom(\mathcal{F},\mathcal{G})^{I}$ and it is natural in $\mathcal{G}$.
\end{definition}

\subsection{Tensor product}
\label{tprod}
The category $\bf{CellRes}$ can be given a tensor product structure.   

\begin{definition} 
\label{tensorprod}
Let $\mathcal{F}$ and $\mathcal{G}$ be any two cellular resolutions with $\mathcal{F}_i=S^{\beta_{F,i}}$ and $\mathcal{G}_j=S^{\beta_{G,j}}$. 
The tensor product of the two resolutions, $\mathcal{F}\otimes \mathcal{G}$, is given by
 $$(\mathcal{F}\otimes \mathcal{G})_k=\bigoplus_{i+j=k}\mathcal{F}_i\otimes \mathcal{G}_j.$$
 
 The differential $d_{k+1}:(\mathcal{F}\otimes \mathcal{G})_{k+1}\rightarrow (\mathcal{F}\otimes \mathcal{G})_k$ is given by 
the matrix for the standard tensor product of chain complexes, with entries simplified such that each column has greatest common divisor 1.
\end{definition}

As defined above the tensor product can be written as a bifunctor $\otimes:{\bf CellRes}\times{\bf CellRes}\rightarrow{\bf CellRes}$.  Also the modules involved are free modules, so  $\mathcal{F}_i\otimes \mathcal{G}_j=S^{\beta_{F,i}\beta_{G,j}}$, and $x\otimes y\in \mathcal{F}_i\otimes \mathcal{G}_j$ corresponds to the element $(x_1y_1,\dots,x_1y_{\beta_{G,j}},\ldots, x_{\beta_{F,i}}y_{\beta_{G,j}})\in S^{\beta_{F,i}\beta_{G,j}}$. 

\begin{remark}
The definition of the tensor product is almost the same as for chain complexes. Indeed on the object level they are the same but the differentials in the  tensor product of $\mathcal{C}_{\bullet}(\bf{Mod}_S)$ would not give an acyclic complex.
\end{remark}

\begin{proposition}
The labeled cell complex of the tensor product of $\mathcal{F}$ and $\mathcal{G}$ is the join of the complexes of $\mathcal{F}$ and $\mathcal{G}$.
\end{proposition}
\begin{proof}
We can compute the associated cell complex from the defined cellular resolution for the tensor product. We see that the vertices stay the same and that the cell complexes of $\mathcal{F}$ and $\mathcal{G}$ are contained in the tensor product. The new edges are formed to connect vertices of the components and respectively the higher dimensional faces. So this is the join of the  complexes. \end{proof}

\begin{remark}
In the case that the label ideals of $\mathcal{F}$ and $\mathcal{G}$ have coprime generators the differential in the tensor product is just the usual differential of the tensor product of chain complexes. 
\end{remark}

\begin{proposition}
The tensor product defined as above is a cellular resolution.
\end{proposition}
\begin{proof}
Firstly, we know that the chain complex defined by the tensor product is made of free modules (tensor product of free modules is free).
We also have that it is supported on a cell complex, so it has a cellular struture. It remains to show the tensor product is acyclic.
The matrix for the differential consists of $(k+1)\times(k+2)$ submatrices, where the $ji$th matrix, denoted by $\delta_{ji}$, is the map from $S^{\beta_{F,i-1}} \otimes S^{\beta_{G,k+2-i}}$ to $S^{\beta_{F,j-1}} \otimes S^{\beta_{G,k+1-j}}$ and is of the size $\beta_{F,j-1}\beta_{G,k+1-j}\times\beta_{F,i-1}\beta_{G,k+2-i}$. The matrix $\delta_{ji}$ has nonzero entries if and only if $i=j$ or $j=i-1$.
Let us look at the case $i=j$ in more detail.
 The positions of the nonzero entries come from the map $1\otimes d_{k+2-i}^G$ identified with a $\beta_{F,j-1}\beta_{G,k+1-j}\times\beta_{F,i-1}\beta_{G,k+2-i}$ matrix. The image $1\otimes d_{k+2-i}^G$ applied to $x\otimes y\in S^{\beta_{F,i-1}} \otimes S^{\beta_{G,k+2-i}}$ is identified with the element
$$\left(\begin{array}{c}
	x_1\sum_{\alpha=1}^{\beta_{G,k+2-i}}(d_{k+2-i}^G)_{1\alpha}y_{\alpha}\\
	\vdots\\
	x_1\sum_{\alpha=1}^{\beta_{G,k+2-i}}(d_{k+2-i}^G)_{\beta_{G,k+1-i}\alpha}y_{\alpha}\\
	x_2\sum_{\alpha=1}^{\beta_{G,k+2-i}}(d_{k+2-i}^G)_{1\alpha}y_{\alpha}\\
	\vdots\\
	x_{\beta_{F,i-1}}\sum_{\alpha=1}^{\beta_{G,k+2-i}}(d_{k+2-i}^G)_{\beta_{G,k+1-i}\alpha}y_{\alpha}\\
\end{array}\right).$$
This can be seen coming from a matrix where the rows are indexed by $uv$ and the columns by $st$, and the $uv,st$ entry is nonzero if and only if $u=s$ and it is given by $(d_{k+2-i}^G)_{vt}$.

The map $\delta_{ji}$ is given by the matrix with the same row and column index as above, and the entries are zero unless $u=s$. Then the $uv,st$ entry is given by $(d_{k+2-i}^G)_{vt}/gcd((d_{k+2-i}^G)_{vt},g_s)$ where $g_s$ is the $s$th generator of the module $S^{\beta_{F,i-1}}$ coming from the labelling. 
From this form one can see that the acyclicity is preserved in the component of the differential.
We can apply the similar argument to the case $j=i-1$, and also get that it preserver acyclicity.
Therefore, the tensor product resolution is acyclic.\end{proof}

With the defined tensor product for cellular resolutions we have the following result. The reader may refer to Section \ref{catsection} for the categorical definitions.
\begin{proposition}
\label{monoid}
The tensor product defined above gives the category {\bf CellRes} a monoidal structure.
\end{proposition}
\begin{proof}
We take the tensor product as defined in Definition \ref{tensorprod} to be our bifunctor $\otimes: {\bf CellRes}\times{\bf CellRes}\rightarrow{\bf CellRes}$.
Take the void resolution, $E:0\leftarrow S\leftarrow0$, as the object $e$ of the monoidal category. 
Define a natural transformation $\alpha:(-\otimes -)\otimes -\rightarrow -\otimes(-\otimes -)$. For $\alpha$ to be a natural isomorhism we need that $(\mathcal{F}\otimes \mathcal{G})\otimes \mathcal{H}\xrightarrow{\alpha}\mathcal{F}\otimes(\mathcal{G}\otimes \mathcal{H})$ is an isomorphism. Using the definition of the tensor product and that on module level it is the same as for the chain complexes, we have that
$$(\mathcal{F}\otimes \mathcal{G})\otimes \mathcal{H}=\mathcal{F}\otimes(\mathcal{G}\otimes \mathcal{H}).$$
Hence we get that $\alpha$ defines a natural isomorphism. 
 
 Let us consider the natural transformations $\lambda:(E\otimes -)\rightarrow -$ and $\rho:(-\otimes E)\rightarrow -$.
 By the definition of natural transformation we have the commutative diagram for $\lambda$, and for any $\mathcal{F},\mathcal{G}\in{\bf CellRes}$ and any morphism $f:\mathcal{F}\rightarrow \mathcal{G}$,
 
 $$\begin{array}{ccc}
	\mathcal{F}\otimes E&\xrightarrow{\lambda_F} & \mathcal{F}\\
	
	\downarrow &&\downarrow \\

	\mathcal{G}\otimes E&\xrightarrow{\lambda_G}& \mathcal{G}
\end{array}.$$
 
 Since $E_j=S$ if $j=0$ and 0 otherwise, we can compute that $(\mathcal{F}\otimes E)_k=\bigoplus_{i+j=k}\mathcal{F}_i\otimes E_j=\bigoplus_{i+0=k}\mathcal{F}_i\otimes S=\mathcal{F}_k$ for any $\mathcal{F}\in{\bf CellRes}$. It is not hard to see that $\lambda_F$ is an isomorphism in {\bf CellRes}, and that $\lambda$ is a natural isomorphism. The same argument can be used for $\rho$ to show that it is also a natural isomorphism. 
 
 A simple computation shows that the triangle and pentagon equalities are also satisfied.
\end{proof}

\begin{example}
\label{tensorexample}
Let $\mathcal{F}$ be the cellular resolution
$$0\leftarrow S\leftarrow S^3\leftarrow S^3\leftarrow S\leftarrow0$$
coming from the complex $F$ in Figure \ref{tensorexamplepic},
and let $\mathcal{G}$ be the cellular resolution
$$0\leftarrow S\leftarrow S^3\leftarrow S^2\leftarrow 0$$ from the complex $G$ in Figure \ref{tensorexamplepic}.
\begin{figure}

\begin{center}
\includegraphics[scale=1]{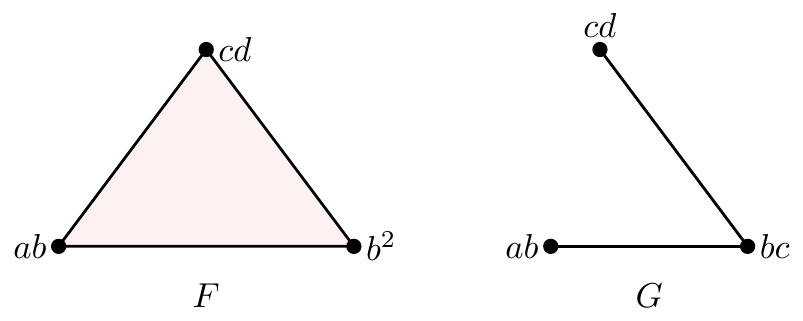}
\end{center}
\caption{Associated cell complexes for the cellular resolutions in Example \ref{tensorexample}.}
\label{tensorexamplepic}
\end{figure}
Then their tensor product is the cellular resolution
$$0\leftarrow S\leftarrow S^3\oplus S^3\leftarrow S^2\oplus S^9\oplus S^3\leftarrow S^6\oplus S^9\oplus S\leftarrow S^3\oplus S^6\leftarrow S^2\leftarrow0$$
with differentials
$$d_1=[ab\ b^2\ cd\ ab\ bc\ cd]$$
$$d_2=\left[\begin{array}{cccccccccccccc}
	-c&0		&1&0&0&b&0&0&cd&0&0		&0&0&0\\
	a&-d		&0&a&0&0&0&0&0&0&0		&0&0&0\\
	0&b		&0&0&ab&0&0&b^2&0&0		&1&0&0&0\\
	
	0&0		&-1&-c&-cd&0&0&0&0&0&0	&-b&cd&0\\
	0&0		&0&0&0&-a&-c&-cd&0&0&0	&a&0&-cd\\
	0&0		&0&0&0&0&0&0&-ab&-b&1	&0&-ab&b^2\\
\end{array}\right]$$
\small{$$d_3=\left[\begin{array}{cccccccccccccccc}
	1&0&b&0&d&0	&0&0&0&0&0&0&0&0&0			&0\\
	0&a&0&b&0&1	&0&0&0&0&0&0&0&0&0			&0\\
	
	c&0&0&0&0&0	&-b&0&0&cd&0&0&0&0&0			&0\\
	-1&d&0&0&0&0	&a&0&0&0&0&0&cd&0&0			&0\\
	0&-1&0&0&0&0	&0&0&0&-ad&0&0&b^2&0&0		&0\\
	0&0&c&0&0&0	&0&-b&0&0&cd&0&0&0&0			&0\\
	0&0&-a&d&0&0	&0&a&0&0&0&0&0&-cd&0			&0\\
	0&0&0&-1&0&0	&0&0&0&0&-ad&0&0&b^2			&0\\
	0&0&0&0&1&0	&0&0&-b&0&0&cd&0&0&0			&0\\
	0&0&0&0&-a&1	&0&0&a&0&0&0&0&0&-cd			&0\\
	0&0&0&0&0&-b	&0&0&0&0&0&-ad&0&0&b^2		&0\\
	0&0&0&0&0&0	&0&0&0&0&0&0&0&0&0			&cd\\
	0&0&0&0&0&0	&0&0&0&0&0&0&0&0&0			&b\\
	0&0&0&0&0&0	&0&0&0&0&0&0&0&0&0			&a
\end{array}\right]$$}

$$d_4=\left[\begin{array}{ccccccccc}
	0&0&0		&-b&0&cd&0&0&0\\
	0&0&0		&0&-b&0&cd&0&0\\
	0&0&0		&a&0&0&0&-cd&0\\
	0&0&0		&0&a&0&0&0&-cd\\
	0&0&0		&0&0&-ab&0&b^2&0\\
	0&0&0		&0&0&0&-ab&0&b^2\\
	
	cd&0&0	&-c&0&0&0&0&0\\
	b&0&0		&0&0&-c&0&0&0\\
	a&0&0		&0&0&0&0&-c&0\\
	0&cd&0	&a&0&d&0&0&0\\
	0&b&0		&0&0&a&d&0&0\\
	0&a&0		&0&0&0&0&a&d\\
	0&0&cd	&0&b&0&0&0&0\\
	0&0&b		&0&0&0&b&0&0\\
	0&0&a		&0&0&0&0&0&b\\
	-ab&-bc&-cd&0&0&0&0&0&0
\end{array}\right]$$
$$d_5=\left[\begin{array}{cc}
	-c&0\\
	a&-d\\
	0&b\\
	cd&0\\
	0&cd\\
	b&0\\
	0&b\\
	a&0\\
	0&a
\end{array}\right]$$

The join of the cell complexes in Figure \ref{tensorexamplepic} is four dimensional cell complex. 

\end{example} 

\section{Limits and colimits}
As with the earlier constructions, we know what the limits are for cell complexes and chain complexes. In the case of limits, and later on colimits, we also know that the categories {\bf Top} and $\mathcal{C}_{\bullet}(\bf{Mod}_S)$ are (co)complete, thus they have all limits and colimits.  However we know that in general limits do not exist in {\bf CellRes} as we do not have the products in general. 
\begin{definition}
A diagram $D$  in $\bf{CellRes}$ is a functor $D: \mathcal{I}\rightarrow {\bf CellRes}$ where $\mathcal{I}$ is a finite indexing category.
\end{definition}

\begin{remark}
When denoting the cellular resolutions in diagrams we use the superscript $D^i$ with $i\in \mathcal{I}$. This is to avoid confusion with the homological degree of the components of the cellular resolution.
\end{remark}

\subsection{Limits}

In general, we can show that the limits in chain complexes as given in Definition \ref{limit}  preserve acyclicity, however just being acyclic is not enough to be a cellular resolution.
\begin{definition}
\label{limit}
Let $D$ be a diagram in $\bf{CellRes}$, and denote the cellular resolutions in it by $D^i$. By definition the limit $L$ of the diagram, if it exists, is the cellular resolution that has a morphism $({\bf f}^i,g^i)$ to each  $D^i$, such that any triangles commute, and $L$ must satisfy the universal property. 
\end{definition}

The product behaviour would suggest problems with the limit when there is non-connected cellular resolutions in the diagram, so for now we restrict ourselves to inverse limits rather than the general limits.

An inverse limit is a limit where the diagram is an inverse system.
\begin{definition}
\label{invsys}
The \emph{inverse system} is given by the following.
Let $(I,\leq)$ be a directed poset. Let $(\mathcal{F}_i)_{i\in I}$ be a collection of cellular resolutions with morphisms $({\bf f},g)_{ij}:\mathcal{F}_j\rightarrow\mathcal{F}_i$ for all $i\leq j$, such that $({\bf f},g)_{ii}$ is the identity and $({\bf f},g)_{ik}=({\bf f},g)_{ij}\circ({\bf f},g)_{jk}$ for all $i\leq j\leq k$.
\end{definition}

We have that for particular class of inverse limits they always exist in {\bf CellRes}.
\begin{proposition}
Let $(\mathcal{F}_i)_{i\in I}$ be a finite inverse system of cellular resolutions such that the underlying poset is a tree. Then 
the inverse limit of $(\mathcal{F}_i)_{i\in I}$ exists in {\bf CellRes}.
\end{proposition}
\begin{proof}
Let $(\mathcal{F}_i)_{i\in I}$ be an inverse system in {\bf CellRes}, and let $r$ be the index of the upper bound element in the poset. 
Let $L$ denote its limit as a chain complex. Then we know that $L_k$ is the inverse limit of modules $(\mathcal{F}_i)_k$, written explicitly as
$$L_k=\lim_{\substack{\leftarrow\\ i\in I}} (\mathcal{F}_i)_k=\left\{(a_1,a_2,\ldots,a_r)\in\prod_{i\in I}(\mathcal{F}_i)_k\ \colon\  a_i=(f_{ij})_k(a_j) \textrm{ for all }i\leq j\in I\right\}.$$
Since the poset is a tree, and $r\geq i$ for all $i\in I$, we can write the module as
$$L_k=\left\{(f_{1r})_k(a_r),(f_{2r})_k(a_r),\ldots,(f_{(r-1)r})_k(a_r),a_r)\in\bigoplus_{i\in I}(\mathcal{F}_i)_k\right\}\cong (\mathcal{F}_r)_k.$$
Furthermore we know that the differentials are the same as in $\mathcal{F}_r$, due to the squares of the maps from $L$ to the diagram being commutative. So we have that $L\cong \mathcal{F}_r$.

Next we want to show that $L$ also satisfies the commutativity requirements for the cellular maps and universal property.
The map from $L$ to $\mathcal{F}_i$ is $({\bf f},g)_{ir}:L\cong \mathcal{F}_r\rightarrow\mathcal{F}_i$. Then by the composition rules for the maps defined in Definition \ref{invsys} we have that for any $i\leq j$, $({\bf f},g)_{ir}=({\bf f},g)_{ij}\circ({\bf f},g)_{jr}$. So $L$ satisfies the commutativity condition of an inverse limit. 
Let $Z$ be a cellular resolution, and suppose that $Z$ maps to every component of $(\mathcal{F}_i)_{i\in I}$. Again any triangles we have must be commutative, so in particular if $\alpha:Z\rightarrow\mathcal{F}_r$ and $\beta: Z\rightarrow \mathcal{F}_ i$, then $\beta=({\bf f},g)_{ir}\circ\alpha$. If $Z$ maps to $L$ all maps must factor through it, in particular $\alpha:Z\rightarrow\mathcal{F}_r$ is then a map composed with identity we get $\alpha:Z\rightarrow L$. With the earlier observation of factoring maps we get that $L$ satisfies the universal property.\end{proof}

\subsection{Colimits}

The situation with colimits is better than with limits. We do infact have (finite) colimits in {\bf CellRes}. The following proposition show the existence and also recalls the definition of colimit.

\begin{proposition}
\label{colim}
Let $D$ be a finite diagram in {\bf CellRes}. Let $C$ be the colimit of the diagram $D$ as colimit of chain complexes with maps ${\bf  f}^i:C\rightarrow D^i$, and let $X$ be the topological colimit of the associated cell complexes in the diagram with maps $f^i:X\rightarrow X^i$. Then $C$ is the cellular complex of $X$,  together with maps $({\bf f}^i,f^i): C\rightarrow D^i$.
\end{proposition}

\begin{proof}
Let $D$ be a finite diagram of cellular resolutions. Let $X^i$ be the cell complex of the $D^i$ cellular resolution, and let $X$ be the topological colimit of the $X^i$'s of the diagram $D$. 

We view $X$ as the space obtained by gluing the cell complexes together with the labels given by Definition \ref{gluedef}. The dimension $k$ faces of $X$ are given as the disjoint union of the dimension $k$ faces of $X^i$'s. One then identifies the $F\in X^i$ with $F'\in X^j$ if there is $f:i\rightarrow j$ such that $D(f)(F)\supseteq F'$.
By definition the cellular free complex of $X$ is given as $\mathcal{F}_k=\bigoplus_{\substack{F\in X\\ dim F=k-1}}S(-{\bf a}_F)$.
We can write the free module as $$\mathcal{F}_k=\bigoplus_{\substack{F\in \coprod X^i\\ dim F=k-1}}S(-{\bf a}_F)/\sim$$ where $S(-{\bf a}_F)\sim S(-{\bf a}_F')$ if $F\sim F'$ due to the gluing observation above. Identifying the free modules with one generator is equivalent to identifying their generators $e_F\sim e_{F'}$. We can write this as
 $$\mathcal{F}_k=\left\{(a_{F_1},a_{F_2},\ldots,a_{F_r})\in\coprod_{\substack{F\in \coprod X^i\\ dim F=k-1}}S(-{\bf a}_F): a_{F_i}\sim a_{F_j}\  \mathrm{  if  }\   e_F\sim e_{F'} \textrm{ and }\right\}.$$
Let $C$ be the chain complex colimit of $D$. By definition of colimit in $\mathcal{C}_{\bullet}(S\textrm{-mod})$, 
$$C_k=\left\{(a_1,a_2,\ldots,a_r)\in \coprod_{i\in I}D^i_k: a_i\sim a_j \mathrm{ if } f_{ij}(a_i)=a_j\right\}.$$
Since the modules come from the cellular resolutions, we know that $D^i_k=\bigoplus_{F\in X^i}S(-{\bf a}_F)$. If there is a map $D^i_k\rightarrow D^j_k$ then we know that the generators map to generators (with multiple $\pm 1$). Thus
$$C_k=\left\{(a_1,a_2,\ldots,a_r)\in \coprod_{i\in I}\bigoplus_{F\in X^i}S(-{\bf a}_F): a_i\sim a_j \mathrm{ if } f_{ij}(e_F)=e_{F'}\right\}.$$
Now it is not hard to see that $C_k$ and $F_k$ are the same module. Therefore we have that $C$ is the cellular complex of $X$.

As the cellular free complex is both the colimit in {\bf Top} and $\mathcal{C}_{\bullet}(S\textrm{-mod})$ we get that it inherits the colimit structure from those categories. What remains to show is that it in fact is a cellular resolution, which means we want to show it is acyclic.
Let $L$ denote the colimit as given above. As with the limit we can draw a diagram with $D^i$ and $D^j$ cellular resolutions shown in Figure \ref{colimpic}. 
\begin{figure}

\includegraphics[scale=1]{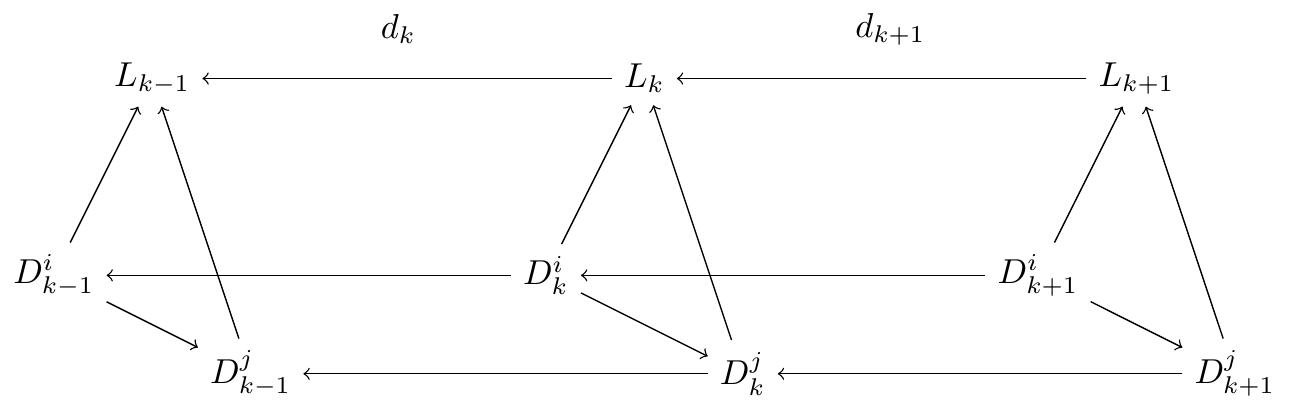}
\caption{Diagram of the maps in colimit.}
\label{colimpic}
\end{figure}
Let $x\in \mathrm{ker }\ d_k$. Using diagram chasing on the diagram in the Figure \ref{colimpic}, we get that $x\in \mathrm{im }\ d_{k_1}$. 
Therefore $L$ is acyclic and hence a cellular resolution.\end{proof}

As a natural corollary to the existence of colimit we have the following.
\begin{corollary}
{\bf CellRes} is a finitely cocomplete category.
\end{corollary} 
\begin{proof} As the above Proposition \ref{colim} holds for any finite diagram of cellular resolutions that has both limit as chain complex and cell complex, we get that we have all finite colimits in {\bf CellRes} since {\bf Top} and $\mathcal{C}_{\bullet}(S\textrm{-mod})$ contain all colimits. Thus {\bf CellRes} is a finitely cocomplete category.\end{proof}

\section{Homotopy colimits}
\label{hocolim}


We begin with an example of homotopy colimit for the cell complexes.
\begin{example}
\label{hocolimex1}
Let $D$ be a diagram of three cellular resolutions $F,G$ ad $H$ with the cell complexes as in the Figure \ref{hcolim}.
\begin{figure}
\begin{center}
\includegraphics[scale=1]{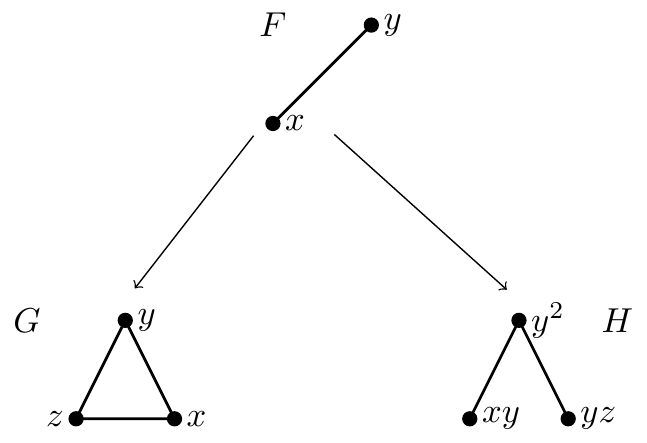}
\end{center}
\caption{Diagram $D$.}
\label{hcolim}
\end{figure}
The resolutions are a follows: $F$ is $0\leftarrow S\leftarrow S^2\rightarrow S^2\leftarrow 0$, $G$ is $0\leftarrow S\leftarrow S^3\leftarrow S^3\leftarrow S\leftarrow 0$, and $H$ is $0\leftarrow S\leftarrow S^3\leftarrow S^2\leftarrow 0$.
The map from $F$ to $G$ is the identity embedding and the map from $F$ to $H$ is the embedding by multiplying with $y$.
Computing the topological homotopy colimit by gluing in mapping cones we get the labeled cell complex in the Figure \ref{hcolim2}. 
This cell complex has the cellular complex
$$0\leftarrow S\leftarrow S^8\leftarrow S^{10}\leftarrow S^3\leftarrow 0$$
with maps 
$$d_0=[x\ y\ x\ y\ z\ xy\ y^2\ yz],$$

 $$d_1=\left[\begin{array}{cccccccccc}
-y& 0&0&0&0&0&-1&0&-y&0\\
x&0&0&0&0&0&0&-1&0&-y\\
0&-z&0&-y&0&0&1&0&0&0\\
0&0&-z&x&0&0&0&1&0&0\\
0&x&y&0&0&0&0&0&0&0\\
0&0&0&0&-y&0&0&0&1&0\\
0&0&0&0&x&-z&0&0&0&1\\
0&0&0&0&0&y&0&0&0&0\\
\end{array}\right],\ \mathrm{ and}$$

$$d_2=\left[\begin{array}{ccc}
0&-1&y\\
-y&0&0\\
x&0&0\\
z&1&0\\
0&0&-1\\
0&0&0\\
0&y&0\\
0&-x&0\\
0&0&-y\\
0&0&x
\end{array}\right].$$
From the maps we can see that the cellular complex is acyclic, hence a resolution.
\begin{figure}

\begin{center}
\includegraphics[scale=1]{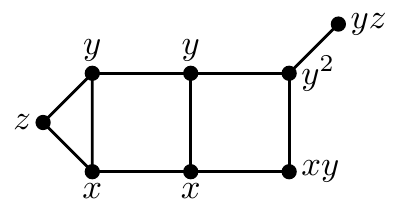}
\end{center}
\caption{Homotopy colimit of the diagram $D$.}
\label{hcolim2}
\end{figure}
\end{example}

The previous example motivates us to lift the gluing mapping cylinders definition to {\bf CellRes}.

\begin{definition}
\label{gluehocolim}

Let $D$ be a diagram in {\bf CellRes} with finite indexing category $\mathcal{I}$. Let $D^i$ and $D^j$ be resolutions in $D$, and let us denote by $\mathfrak{f}_{ij}=({\bf f}^{ij}, f^{ij})$ the morphism between them if we have a map $\psi:i\rightarrow j$ in $\mathcal{I}$.
We have a mapping cylinder for each morphism $\mathfrak{f}_{ij}$. The $k$th piece of the mapping cylinder is given by $D^i_k\oplus D^j_k\oplus M^{ij}_k$, with differentials $d_k(a_i,a_j,m)=(d^i_k(a_i)+f^{ij}_i(m),d^j(a_j)+f^{ij}_j(m),f^{ij}_{ij}(m))$.

Then the homotopy colimit obtained by gluing in mapping cylinders to the coproduct  is the resolution with $k$-th piece 
$$\left(\oplus_{i\in\mathcal{I}}D^i_k\right)\oplus\left(\oplus_{i\rightarrow j\in\mathcal{I}}M^{ij}_k\right)$$
with differentials
$$d_k(a_1,a_2,\ldots,a_r,m_{11},m_{12},\ldots,m_{rr})=(d^i_k(a_i)+\sum_{i\rightarrow j}f^{ij}_i(m),\sum f^{ij}_{ij}(m)).$$

\end{definition}

Recall that Proposition \ref{gluemap} states that gluing in mapping cylinders into a diagram is a cellular resolution. Therefore the above definition also is a cellular resolution.

The homotopy colimit as defined in Definition \ref{gluehocolim} can easily give a very large cellular resolution that is far from minimal. However, it can be homotopy equivalent to a smaller one. For instance in Example \ref{hocolimex1} we could remove the middle square without changing any important properties. 

Recall from topology that the homotopy colimit can be defined as the direct sum $\cup_{a\in \mathcal{I}}B(I\downarrow a)\times D_a$ quotient by some relations (see Definition \ref{tophocolim}).

This then gives a second definition for the homotopy colimit in {\bf CellRes}. First we need to define the geometric realization in {\bf CellRes}.
\begin{definition}
Let $X$ be a simplicial set. We define the \emph{geometric realization} of $X$ in {\bf CellRes} to be the free resolution coming from the geometric realization of $X$ in {\bf Top} with labels 1 on each vertex. 
\end{definition}

\begin{definition}
\label{prodhocolim}
Let $\mathcal{I}$ be a finite small category and let $D$ be a diagram in {\bf CellRes}. Let $D^i$ and $D^j$ be resolutions in $D$, and let the morphism between them be denoted by $\mathfrak{f}_{ij}=({\bf f}^{ij}, f^{ij})$  if we have a map $\psi:i\rightarrow j$ in $\mathcal{I}$. Then we define the \emph{homotopy colimit} to be the direct sum
$$\sqcup B(i\downarrow \mathcal{I})\times D^i$$
 quotient by a relation $\sim$. Here $B(i\downarrow \mathcal{I})$ is the geometric realization of the nerve of the category under $i$, $(\mathcal{I}\downarrow i)$, and $D^i$ is the element in the diagram $D$ associated to the element $i\in \mathcal{I}$.
We have the maps $\mathfrak{a}:B(j\downarrow \mathcal{I})\times  D^i\rightarrow B(j\downarrow \mathcal{I})\times  D^j$ and $\mathfrak{b}:B(j\downarrow \mathcal{I})\times  D^i\rightarrow B(i\downarrow \mathcal{I})\times  D^i$, given by 
$$\mathfrak{a}(p,x)=(p,\mathfrak{f}_{ij}(x))$$
and
$$\mathfrak{b}(p,x)=(\delta_{ji}(p),x)$$
for every map $i\rightarrow j\in \mathcal{I}$, where $\delta_{ji}:B(j\downarrow \mathcal{I})\hookrightarrow B(i\downarrow \mathcal{I})$.
Then the quotient is given by the relation $\mathfrak{a}(p,x)\sim \mathfrak{b}(p,x)$. 

\end{definition}

\begin{example}
\label{hocolimex2}
Let the diagram and cellular resolutions be as in Example \ref{hocolimex1} and Figure \ref{hcolim}. We use the Definition \ref{prodhocolim} to compute the homotopy colimit for the diagram. 

The indexing category $\mathcal{I}$ has three objects, say $a,b$ and $c$ for $F$, $G$, and $H$ respectively. We compute the geometric realization of the nerve of the category under each of the objects.
Starting with $B(a\downarrow \mathcal{I}) $, the simplicial set $(a\downarrow \mathcal{I})$ has chains of length zero coming from the objects $a\rightarrow b$ and $a\rightarrow c$, and $a\rightarrow a$. Then it also has chains of length one from $(a\rightarrow a)\rightarrow (a\rightarrow b)$ and $(a\rightarrow a)\rightarrow (a\rightarrow c)$, and the ones coming from identity maps. The higher degree ones are given by adding identity maps to the length one chains. 
Then $B(a\downarrow \mathcal{I})$ is given as a cell complex by 
$\sqcup_{n\geq 0}(a\downarrow \mathcal{I})_n\times \Delta_n/\sim$. Computing this, we see that the $n>1$ parts coming from the identity maps are identified to a point, and the lower ones give one directed edge for each non-identity map. We get that the cell complex is $\bullet\leftarrow\bullet\rightarrow\bullet$. The resolution for it is $0\leftarrow S\leftarrow S^3\leftarrow S^2\leftarrow 0$. 
Similarly, we get that $B(b\downarrow \mathcal{I})=B(c\downarrow \mathcal{I})=\Delta_0$, and they have the resolution $0\leftarrow S\leftarrow S\leftarrow0$.

Next we have the disjoint union of 
$$B(a\downarrow \mathcal{I})\times F\sqcup \Delta_0\times G\sqcup\Delta_0\times H.$$
Each of the products is with trivially labeled cellular resolution. Thus they exist and are cellular resolutions. We know that the generators of the product, and also the labels of the associated cell complex are coming from the $D^i$ in this case. 
 Then we take the quotient by the relation $\mathfrak{a}_i(p,x)=\mathfrak{b}_i(p,x)$ where
$$\mathfrak{a}_1:\Delta_0\times  F\rightarrow \Delta_0\times  G,$$
$$\mathfrak{b}_1:\Delta_0\times  F\rightarrow B(a\downarrow \mathcal{I})\times  F,$$
$$\mathfrak{a}_2:\Delta_0\times  F\rightarrow \Delta_0\times  H,$$
$$\mathfrak{b}_2:\Delta_0\times  F\rightarrow B(a\downarrow \mathcal{I})\times  F.$$
Note that if we have a map $a\rightarrow b$, then $B(b\downarrow \mathcal{I})\hookrightarrow B(a\downarrow \mathcal{I})$.  Expicitly the identification of the "copies of $F$" in the part of $B(a\downarrow \mathcal{I})$  coming from $B(b\downarrow \mathcal{I})$ are glued to $B(b\downarrow \mathcal{I})\times  G$ along the map $F\rightarrow G$. We do the same gluing for all the maps in $D$.
After the identification we have the cellular resolution 
$$0\leftarrow S\leftarrow S^8\leftarrow S^{10}\leftarrow S^3\leftarrow 0$$
with the same maps as in Example \ref{hocolimex1}. 
\end{example}

\begin{proposition}
Definition \ref{gluehocolim} and Definition \ref{prodhocolim} are equivalent.
\end{proposition}
\begin{proof}
Since both definitions are lifted from {\bf Top}, we know that the underlying cell complexes associated the cellular resolutions are the same. 

Starting with Definition \ref{gluehocolim},
let us consider the case of maps from $D^i$. 
The mapping cylinder for $\mathfrak{f}_{ij}:D^i\rightarrow D^j$ can be thought of as the disjoint union of $\Delta_1\times D^i$ and $\Delta_0\times D^j$ quotient by the relation $(p_1, x)\sim (p, \mathfrak{f}_{ij}(x))$. 
We also need to glue along all the mapping cones containing $D^i$, and if we have a composition then along those as well. 
Writing this as the single mapping cylinder above, we have the disjoint union of $\Delta_1\times D^i$, one for each map from $i$, and a copy of $\Delta_0\times D^j$ for each $j$ $i$ maps to. The quotients are along $(p_1,x)\in \Delta_1\times D^i$ all glued together, $(p_1, x)\sim (p, \mathfrak{f}_{ij}(x))$ for each map, and if we have a map $j\rightarrow k$ then we identify $(p,\mathfrak{f}_{ik}(x))$ with $(p,\mathfrak{f}_{jk}(\mathfrak{f}_{ij}(x)))$. This is the same as taking the product of $D^i$ with simplicial space that records the map structure, i.e. precisely $B(i\downarrow \mathcal{I})$, and gluing the vertices of $B(i\downarrow \mathcal{I})$ corresponding to the map $\mathfrak{f}_{ij}$ with the $\Delta_0$ of $\Delta_0\times D^j$ while identifying $D^i$ and $D^j$ via $\mathfrak{f}_{ij}$.

Taking the above product for all the $D^i$ in the diagram, and gluing them together by the mapping cylinder rules then gives us that we have gluing between  $B(i\downarrow \mathcal{I})\times  D^i$ and $B(j\downarrow \mathcal{I})\times  D^j$ if we have a map $i\rightarrow j\in\mathcal{I}$. Then the identification of the gluing is given by 
$$(p,\mathfrak{f}_{ij}(x))\sim =(\delta_{ji}(p),x)$$
which recovers the relation $\mathfrak{a}(p,x)\sim\mathfrak{b}(p,x)$. 
Therefore it is equivalent to Definition \ref{prodhocolim}, and the two definitions give the same cellular resolution.
Moreover, this shows that Definition \ref{prodhocolim} does give a cellular resolution.\end{proof}

\section{Morse theory and simple homotopy theory}
\label{morsesec}

\subsection{Morse theory for cellular resolutions}
In this section we want to look at Morse theory for cellular resolutions. Recall that the idea of discrete Morse theory is to collapse cells to reach a smaller cell complex with the same homology, and similarly in algebraic morse theory.  In the case of cellular resolutions, Morse theory forms an useful tool "remove" the non-minimal part of the resolution by collapsing it. We illustrate this in the following example.

\begin{example}
\label{exmorse1}
Let $X$ be the filled in triangle with vertex labels $ab,bc,cd$, see Figure \ref{tensorexamplepic} cell complex $F$. The $X$ has a Taylor complex
$$F:0\leftarrow S\leftarrow S^3\leftarrow S^3\leftarrow S\leftarrow 0.$$
Both the face poset of $X$ and the graph $\Gamma_F$, defined in Section \ref{algmorse}, have the directed graph shown in Figure \ref{graph}.

\begin{figure}
\begin{center}
\includegraphics[scale=1]{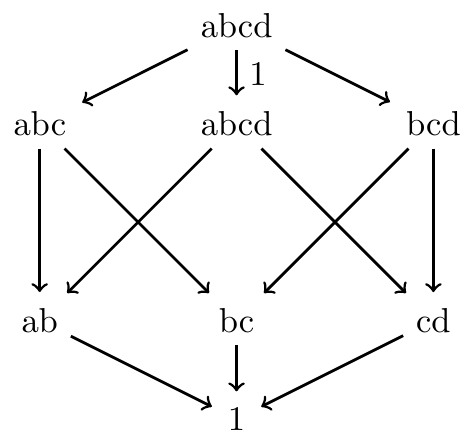}
\end{center}
\caption{Graph of the face poset and $\Gamma_F$ of Example \ref{exmorse1}. Here we have used the monomial labels to denote both the cell with that label and the free module with generator associated to that label. }
\label{graph}
\end{figure}

 Let us take the single edge marked with 1 as the Morse matching.  Then we can compute the chain complex coming from this with algebraic Morse theory
 $$0\leftarrow S\xleftarrow{d_1} S^3\xleftarrow{d_2} S^2\leftarrow 0$$
  with differentials
  $$d_1=\left[ ab\ bc\ cd\right],$$
  $$d_2=\left[ \begin{array}{cc}
  -c&0\\ 
  a&-d\\ 
  0&b\\
  \end{array}\right].$$
 This resolution has the cell complex $G$ of Figure \ref{tensorexamplepic} as its cell complex. This is exactly the same cell complex one gets by doing discrete Morse theory on $X$.
\end{example}

 The first results on the discrete  Morse theory of cellular resolutions were in paper by Batzies and Welker \cite{BW}, and they have been later applied in other works, for example in \cite{MAG} to construct an algorithm for finding cellular resolutions closer to the minimal one.

One of the main concerns when applying Morse theory to cellular resolutions is whether the resulting complex is still a cellular resolution. In a case with restricted collapses this was solved by Batzies and Welker.
\begin{theorem}[\cite{BW}]
\label{bwthm}
Let $X$ be a complex that supports a cellular resolution, and let $M$ be a Morse matching on this complex. If $M$ only matches cells with the same labels, then the Morse complex $\tilde{X}$ also supports a cellular resolution of the same ideal.
\end{theorem}

Only matching cells with the same label is a strong restriction on what kind of matchings we can make. Algebraically this is the same as only choosing isomorphisms for the matching, which is indeed the case in algebraic Morse theory as we have seen in the Section \ref{algmorse}.

\begin{proposition}
Let $\mathcal{F}$ be a cellular resolution with a cell complex $X$. Let $M$ be a Morse matching on the face poset of $X$ (or the graph $\Gamma_{\mathcal{F}}$). Suppose that $M$  only matches cells with the same labels. Then $M$ is a Morse matching also on $\Gamma_F$ (or on the face poset of $X$). 
Furthermore, $\tilde{\mathcal{F}}$ is the cellular complex of $\tilde{X}$.
\end{proposition}
\begin{proof}
Note that the face poset $P_X$ of $X$ and the graph $\Gamma_{\mathcal{F}}$ of $\mathcal{F}$ are the same directed graph up to the labels on the vertices. The vertex in $P_X$ corresponding to cell $x$ is the same vertex in $\Gamma_{\mathcal{F}}$ with the free module having generator corresponding to $x$.

Let $M$ be a Morse matching on $P_X$, and suppose it only matches cells with the same labels. 
We can consider the matching $M$ on $\Gamma_{\mathcal{F}}$. Since the underlying graph is the same as in $P_X$, $M$ is a  matching on $\Gamma_{\mathcal{F}}$ without any directed cycles when the edges in $M$ are reversed. 
The vertices in $\Gamma_{\mathcal{F}}$ come from the summands at each homological degree, so by definition a vertex is $S(-{\bf a}_z)$ where ${\bf a}_z$ is the fine graded degree of the label of some $z\in X$.  By the same label condition for edges in $M$, we have that if the edge $E\in M$ matches $S(-{\bf a}_x)$ and $S(-{\bf a}_y)$, then the cells $x$ and $y$ have the same label. This gives that ${\bf a}_x={\bf a}_y$.
Then the corresponding map in $\mathcal{F}$ between $S(-{\bf a}_x)$ and $S(-{\bf a}_y)$ is an isomorphism, as they are the same free module. It follows that $M$ is a Morse matching on $\Gamma_{\mathcal{F}}$ as well.

On the other had, if $M$ is a Morse matching on $\Gamma_{\mathcal{F}}$, we again get directly that it is matching on $P_X$ without directed cycles of the edges in $M$ are reversed. $M$ also only matches free modules that are isomorphic, and we know that if $S(-{\bf a}_x)\cong S(-{\bf a}_y)$ then $x$ and $y$ have the same label. So the corresponding vertices in $P_X$ of a matched edge in $M$ are cells with the same label. Thus $M$ is a Morse matching on $P_X$ that only matches cells with the same labels.

We have established that the algebraic and same label cell matchings are the same. Let $\tilde{\mathcal{F}}$ be the homotopic chain complex of $\mathcal{F}$ after the Morse map and let $\tilde{X}$ be the cell complex after collapses on $X$.  One way to see that $\tilde{\mathcal{F}}$ is the cellular complex of $\tilde{X}$ is to consider their poset diagrams. Since they come from the posets of $\mathcal{F}$ and $X$, we know that the vertices correspond to each other, and the remaining vertices have not changed. As the same Morse matching $M$ is applied to both $\mathcal{F}$ and $X$, the resulting posets are the same up to vertex labels. Lastly we know from Theorem \ref{bwthm} that $\tilde{X}$ supports a resolution, and as $\tilde{\mathcal{F}}$ is the cellular complex of $\tilde{X}$, we have that $\tilde{\mathcal{F}}$ is a cellular resolution.
\end{proof}

\begin{theorem}

Let $\mathcal{F}$ be a cellular resolution with a cell complex $X$, and let $M$ be a Morse matching on them. 
Let ${\bf f}$ be the chain map from $\mathcal{F}$ to $\tilde{\mathcal{F}}$, and let $f$ be the cellular strong deformation retract of $X$ coming from Morse theory.
Then the pair $({\bf f},f)$  formed of the Morse maps is a morphisms in {\bf CellRes}.
\end{theorem}
\begin{proof}
A single edge in the Morse matching corresponds to a collapse, and we can do these collapses one by one. Thus we may assume that the Morse matching in this case is a single edge.
We want to show that {\bf f} and $f$ are compatible, so that they form a morphism in {\bf CellRes}. 

Let $M$ be a Morse matching with a single edge that generates the maps. Let us denote by $e_i$ and $e_{i+1}$ the generators of the vertices in $\Gamma_{\mathcal{F}}$, and with $x_i$ and $x_j$ the cells corresponding to the vertices of $P_X$.
For any unmatched vertices, in both $\Gamma_{\mathcal{F}}$ and $P_X$, the maps {\bf f} and $f$ act on these vertices as an identity. It follows from the definition of the identity that the maps satisfy the compatibility for those parts. In particular, the chain map ${\bf f}$ has $f_k=\operatorname{id}$ for any $k<i$, and cells of lower dimension than $i$ map by identity.

Let us focus on the maps $f_i$, $f_{i+1}$ and $f$ on cells of dimension $i$ and $i+1$. 
We will first show that $f_i(e_i)$ is a linear combination of other generators of the modules $e_{i+1}$ maps to.
Since we have explicit description of the resolution $\tilde{\mathcal{F}}$, we have that $$f_i: \bigoplus_{\substack{x\in X\\ dim x= i-1}} S(-a_x)\rightarrow \bigoplus_{\substack{x\notin M\\ dim x=i-1}} S(-a_x).$$
As noted before, $f_i$ is identity on $\bigoplus_{x\notin M} S(-a_x)\subset\bigoplus_{x\in X} S(-a_x)$. So we only need to compute $f_i(e_i)$.  From the definition of the chain map we have that $d\circ f_i(e_i)=f_{i-1}\circ d(e_i)=id\circ d(e_i)$.  Let $y_1,\ldots, y_r$ be the generators of $ S(-a_x)$ where $x\notin M$ and there is an edge from $x_{i+1}$ to $x$.  From the differentials and their composition we have that $d(e_i)=d(a_1y_1+\ldots+a_ry_r)$ for some $a_1,\ldots,a_r\in S$. Combining this with the commutativity, we get that $f_i(e_i)=a_1y_1+\ldots+a_ry_r$. 
Similarly, we can show that $f_{i+1}(e_{i+1})=0$ using the properties of differentials and the commutative squares.

On the topological side, $f$ deformation retracts $x_{i+1}$ and $x_i$ to $\tilde{X}$. So they both map to the intersection of boundary of $x_{i+1}$ and $\tilde{X}$. 
Then $x_{i+1}$ will map to a smaller dimensional cells, so a compatible chain map maps the generator associated to $x_{i+1}$ to 0 in degree $i+1$, which we have with $f_{i+1}$.
These conditions also give that $x_i$ either maps to cells of the same dimension or to lower ones. In the former case, we get that these are the cells corresponding to the algebraic generators $y_1,\ldots, y_r$, and we have the compatibility that we want.
In case $f(x_i)=0$, we then know there are no other $i$-cells on the boundary of $x_{i+1}$. This also implies that $e_{i+1}$ only maps to one element, $e_i$. Thus we get that $f_i(e_i)=0$, so the two maps are compatible.
\end{proof}

\subsection{Simple homotopy theory}
Having the morse maps as morphisms in {\bf CellRes} allows us to define simple homotopy for cellular resolutions. For classical simple homotopy theory the reader may look up the book by Cohen \cite{cohen}.

In the language of simple homotopy theory, a Morse matching with a single edge is an \emph{elementary collapse}.  We have also the elementary expansion.

\begin{definition}
Let $\mathcal{F}$ and $\mathcal{G}$ be cellular resolutions. Suppose that there is an elementary collapse from $\mathcal{F}$ to $\mathcal{G}$. Then we say that $\mathcal{G}$ \emph{expands to $\mathcal{F}$ by an elementary expansion}, or that there is an \emph{elementary expansion} from $\mathcal{G}$ to $\mathcal{F}$.
A finite sequence of elementary collapses and elementary expansions is called a \emph{formal deformation}.
\end{definition}

Now we can define the simple homotopy equivalence for cellular resolutions.
\begin{definition}
Let $\mathcal{F}$ and $\mathcal{G}$ be cellular resolutions. Let $\mathfrak{f}=({\bf f},f):\mathcal{F}\rightarrow \mathcal{G}$ be a formal deformation. Then we get that $\mathcal{F}$ and $\mathcal{G}$ are simple homotopy equivalent.
\end{definition}

\end{document}